\documentclass[12pt]{amsart}

\usepackage{amsfonts}
\usepackage{amsmath}
\usepackage{amssymb}

\usepackage{mathrsfs}
\usepackage{amscd}
\usepackage[all]{xy}
\usepackage[pdftex,final]{graphicx}
\usepackage{hyperref}

\newtheorem{lem}{Lemma}[section]
\newtheorem{thm}[lem]{Theorem}
\newtheorem{prop}[lem]{Proposition}
\newtheorem{cor}[lem]{Corollary}
\theoremstyle{definition}

\newtheorem{exa}[lem]{Example}

\newtheorem{rem}[lem]{Remark}

\newcommand{\Q}{\Bbb{Q}}
\newcommand{\F}[1]{\Bbb{F}_{#1}}

\newcommand{\Z}{\Bbb{Z}}

\newcommand{\R}{\Bbb{R}}

\newcommand{\wn}[1]{\mathcal{W}_{#1}}% set of units u such that 1-u is a unit 

\newcommand{\spl}[2]{\operatorname{SL}_{#1}\left(#2\right)}

\newcommand{\gl}[2]{\mathrm{GL}_{#1}(#2)}
\newcommand{\pgl}[2]{\mathrm{PGL}_{#1}(#2)}

\newcommand{\pn}[2]{\mathbb{P}^{#1}(#2)}
\newcommand{\projl}[1]{\pn{1}{#1}}

\renewcommand{\ker}[1]{\mathrm{Ker}(#1)}

\newcommand{\an}[1]{\left\langle{#1}\right\rangle}

\newcommand{\matr}[4]{\left[\begin{array}{cc}
#1&#2\\
#3&#4\\
\end{array}
\right]}

\newcommand{\st}[2]{\mathrm{St}(#1,#2)}

\newcommand{\St}[2]{\mathrm{St}(#1,#2)}

\newcommand{\stand}[1]{\mathrm{st}(#1)}

\title{$\mathrm{GE}_2$-Rings and a graph of unimodular rows}
\author{  Kevin Hutchinson}
\address{School of Mathematics and Statistics,
 University College Dublin}
\email{  kevin.hutchinson@ucd.ie}
\date{\today}


\begin{document}
\maketitle
%\tableofcontents
\begin{abstract}
For a commutative ring $A$ we consider a related graph, $\Gamma(A)$, whose vertices are the unimodular rows of length $2$ up to multiplication by units.
We prove that $\Gamma(A)$ is path-connected if and only if $A$ is a $\mathrm{GE}_2$-ring, in the terminology of P. M. Cohn. Furthermore, 
if $Y(A)$ denotes the clique complex of $\Gamma(A)$, we prove that $Y(A)$ is simply connected if and only if $A$ is  universal for $\mathrm{GE}_2$.
More precisely, our main theorem is that for any commutative ring $A$ the fundamental group of $Y(A)$ is isomorphic to the group  $K_2(2,A)$ modulo the subgroup 
generated by symbols. 
\end{abstract}
\section{Introduction}
For a commutative ring $A$ we consider the following functorially associated graph $\Gamma(A)$: The vertices of $\Gamma(A)$ are unimodular rows $(a,b)$ up to multiplication by units.
Two (equivalence classes of) rows form an edge of $\Gamma(A)$ if they are the rows of a matrix in $\gl{2}{A}$. Let $Y(A)$ be the clique complex of $\Gamma(A)$. Thus $Y(A)$ is the simplicial complex whose set of $n$-simplices $Y_n(A)$ is the set of $n$-cliques (complete subgraphs on $n$ vertices) of $\Gamma(A)$. For reasons detailed  below in this introduction, we wished to determine the class of
 rings $A$ for which $Y(A)$ is homologically $1$-connected and the investigations in this paper were motivated by this question.

By considering paths from $\infty:=(1,0)$, it is not too difficult to see that the space $Y(A)$, or equivalently the graph $\Gamma(A)$, is path-connected if and only if $\spl{2}{A}$ is generated 
by elementary matrices; i.e., if and only if $A$ is a $\mathrm{GE}_2$-ring in the terminlogy of P. M. Cohn (\cite{cohn:gln}). More precisely, we show (Theorem \ref{thm:pi0})  that the set of path components,
$\pi_0(Y(A))$, of $Y(A)$ is naturally in bijective correspondence with the coset space $E_2(A)\backslash \spl{2}{A}$, where $E_2(A)$ is the subgroup generated by elementary matrices. 
We recall  that while every Euclidean Domain is a $\mathrm{GE}_2$-ring, there are many examples of PIDs which are not $\mathrm{GE}_2$-rings. In Section \ref{sec:path} below we also 
show how paths from $\infty$ in $\Gamma(A)$ correspond to weak Euclidean algorithms for unimodular pairs and, when $A$ is an integral domain with field of fractions $F$, to the existence of continued 
fraction expansions of elements of $F$. 

The fundamental group of $Y(A)$, on the other hand, is, as we show,  related to Cohn's notion of rings which are \emph{universal for $\mathrm{GE}_2$}. The group $E_2(A)$ is generated by the matrices $E(a):=\left[
\begin{array}{cc}
a&1\\
-1&0\\
\end{array}
\right], a\in A$
and Cohn (\cite{cohn:gln})  writes down certain universal relations satisfied by these matrices. The ring $A$ is, by definition,  \emph{universal for $\mathrm{GE}_2$}, if these generators and relations give a presentation of the group $E_2(A)$. In section \ref{sec:univ} below, we consider the group $C(A)$ described by this presentation (in the case of \emph{commutative} $A$) and the kernel
$U(A)$,  of the surjective homomorphism $C(A)\to E_2(A)$. Thus, by definition, $A$ is universal for $\mathrm{GE}_2$ if and only if $U(A)=1$. Our main theorem (Theorem \ref{thm:main}) gives an explicit isomorphism, with explicit inverse, from the fundamental group $\pi_1(Y(A),\infty)$ to the group $U(A)$. It follows that $Y(A)$ is simply connected if and only if 
$A$ is universal for $\mathrm{GE}_2$. The essential work in the article consists in preparing the ground for the
proof of the main  theorem by writing down a presentation for this fundamental group.

In fact, it is well-known that, for a commutative ring $A$, the condition of being universal for $\mathrm{GE}_2$ is equivalent to the condition that the rank one $K_2$ group 
$K_2(2,A)$ is generated by symbols; i.e., that $K_2(2,A) = C(2,A)$ where $C(2,A)$ is the subgroup generated by the set of symbols $\{ c(u,v)\ |\ u,v\in A^\times\}$. However, I am not aware that a proof of this result has been published.  For example, Dennis and Stein state this fact on page 228 of \cite{dennisstein:dvr} and refer to unpublished notes for the proof. Other articles which use this result generally cite Dennis-Stein. Since we  use this result in an essential way below, we have added, for our own convenience and that of the reader, an appendix in which we verify that there is a natural isomorphism $U(A)\cong K_2(2,A)/C(2,A)$
for any commutative ring $A$ (Theorem \ref{thm:gamma}). In view of this isomorphism, our main theorem also states that the fundamental group of $Y(A)$ is isomorphic to 
$K_2(2,A)/C(2,A)$.

We now describe our original motivation for studying the questions addressed in this article. To a commutative ring $A$ we associate a complex of abelian groups $L_\bullet(A)$ as follows:
Let $X_n(A)$ denote the set all ordered $(n+1)$-tuples $(x_0,\ldots,x_n)$ where $\{ x_0,\ldots,x_n\}\in Y_n(A)$ is an $(n+1)$-clique of $\Gamma(A)$. Let $L_n(A):=\Z[X_n(A)]$ and 
let $d_n: L_n(A)\to L_{n-1}(A)$ be the standard simplicial boundary homomorphism. We also have an augmentation $\epsilon:L_0(A)\to \Z, (x_0)\mapsto 1$.  The complex $L_\bullet(A)$ is naturally a complex of right modules over the group $\spl{2}{A}$, and has proved very useful in calculating the low-dimensional homology of this group, when $A$ is a field or local ring,  in terms of refined Bloch groups (\cite{hut:rbl11},\cite{hut:slr},\cite{hut:sl2Q}). With R. C. Coronado, we have investigated the possibility of extending these results and methods  to more general rings,
such as local rings with small residue fields, principal ideal domains and rings of integers  (\cite{corohut:bloch}). We have found that in order for the Bloch groups of a ring, appropriately defined,  to have the desired relationship to $H_3(\spl{2}{A},\Z)$ or to the indecomposable $K_3$ of $A$,  we must have that the complex
\[
L_2(A)\to L_1(A)\to L_0(A)\to \Z\to 0
\]
is exact; i.e., that $H_0(L_\bullet(A))\cong \Z$ and $H_1(L_\bullet(A))=0$. However,  there is a natural map of complexes $L_\bullet(A)\to C_\bullet(Y(A))$ (the oriented chain complex of $Y(A)$) and it is not hard to see that this induces an isomorphism on homology in dimensions $0$ and $1$\footnote{On the other hand, $H_2(L_\bullet(\Z))\not=0$ while $H_2(C_\bullet(Y(\Z)))=H_2(Y(\Z))=0$.}. Thus our question becomes: for which rings $A$ is $Y(A)$ homologically $1$-connected. Since $H_1(Y(A))\cong \pi_1(Y(A),\infty)^{\mathrm{ab}}$ 
our main theorem gives an answer to this question in terms of the $K$-theory group $K_2(2,A)$. The $K_2$ calculations of Morita in \cite{morita:k2zs} show that the Euclidean domains $\Z[\frac{1}{m}]$ are universal for $\mathrm{GE}_2$ for many integers  $m$ which are divisible by $2$ or $3$ (Example \ref{exa:pi1m}). However, Morita's calculations in \cite{morita:mab} can be used to show that for any prime $p\geq 5$, $H_1(L_\bullet(\Z[\frac{1}{p}])\not=0$ and they even allow us to write down an explicit short $1$-cycle in $L_1(\Z[\frac{1}{p}])$ which represents a homology class of infinite order (Example 
\ref{exa:pi1p}). 

\textbf{Acknowledgement.} We thank the anonymous referee for a number of valuable suggestions and corrections which have been incorporated into the current version of the article.

\subsection{Notation and conventions}
\emph{All rings below are commutative.}  For a ring $A$, $A^\times$ denotes its group of units. 

For a simplicial complex $Y= \{ Y_n\}_n$, $|Y|$ will denote its geometric realization with the weak topology.
%%%%%%%%%%%%%%%%%%%%%%%%%%%%%%%%%%%%%%%%%%%%%%%%%%%%%%%%%%%%%%%%%%%%%%%%%
\section{The  graph $\Gamma(A)$}
%%%%%%%%%%%%%%%%%%%%%%%%%%%%%%%%%%%%%%%%%%%%%%%%%%%%%%%%%%%%%%%%%%%%%%%%%%%
%We begin by recalling some terminology and results from section 3 of \cite{hut:slr}.  
Let $A$ be a  ring and let $\Gamma(A)$ be the following associated graph: The vertices of $\Gamma(A)$ are equivalence classes, $[u]$, of unimodular rows $u=(u_1,u_2)\in A^2$ under scalar multiplication by units in $A$. The pair $\{ [u],[v]\}$ is an edge in $\Gamma(A)$ if the matrix 
\[
M=
\left[
\begin{array}{c}
u\\
v\\
\end{array}
\right]
\]
lies in $\gl{2}{A}$; i.e., if $\mathrm{det}(M)\in A^\times$. Let $Y_0(A)$ denote the set of vertices of $\Gamma(A)$ and let $Y_1(A)$ denote the set of edges. Observe that both of these sets are naturally right $\gl{2}{A}$-sets (or, indeed, $\pgl{2}{A}$-sets) via right multiplication by matrices. 

More generally, we let $Y(A)$ denote the \emph{clique complex} of the graph $\Gamma(A)$. This is the simplicial complex whose set of  $n$-simplices, $Y_n(A)$,  is the set of  $(n+1)$-cliques of $\Gamma(A)$. Recall that an $n$-clique is a set of $n$ distinct vertices $\{ [u_1],\ldots, [u_n]\}$  with the property that every pair is an edge of $\Gamma(A)$.  Let $X_n(A)$ denote the set all ordered $(n+1)$-tuples $(x_0,\ldots,x_n)$ where $\{ x_0,\ldots,x_n\}\in Y_n(A)$ is an $(n+1)$-clique of $\Gamma(A)$.

%%%%%%%%%%%%%%%%%%%%%%%%%%%%%%%%%%%%%%%%%%%%%%%%%%%%%%%%%%%%%%%%%%%%%%%%%%%
\subsection{The vertices of $\Gamma(A)$}
%%%%%%%%%%%%%%%%%%%%%%%%%%%%%%%%%%%%%%%%%%%%%%%%%%%%%%%%%%%%%%%%%%%%%%%%%%%%
We will use the following notations: For any $a\in A$, $a_+$ denotes (the class of)  $(a,1)$ in $Y_0(A)$ and $a_-$ denotes the class of $(1,a)$. Thus if $a$ is a unit then $a_-=(a^{-1})_+$ in $Y_0(A)$. Furthermore, we set
\[
0:=0_+=(0,1),\quad \infty: = 0_-=(1,0),\quad 1:=1_+=1_-,\quad -1:= (-1)_+=(-1)_- \quad \mbox{ in } Y_0(A).
\] 

The action of $\gl{2}{A}$ on $Y_0(A)$ is transitive: If $u$ is any unimodular row then, by definition, there exists a unimodular row $v$ such that 
$
X:=\left[
\begin{array}{c}
u\\
v\\
\end{array}
\right]
\in \gl{2}{A}.
$
 Then $\infty\cdot X=[u]$. The stabilizer of $\infty$ is the subgroup $\tilde{B}=B(\gl{2}{A})$ of lower triangular matrices in $\gl{2}{A}$. Furthermore, $\spl{2}{A}$ acts transitively on the vertices of $\Gamma(A)$ and the stabilizer of $\infty$ is $B=B(\spl{2}{A}:=\tilde{B}\cap\spl{2}{A}$.

Now  let $A$ be an integral domain with field of fractions $F$. We will say that $(0,0)\not=(a,b)\in A\times A$ is a \emph{B\'ezout pair} if the ideal $\an{a,b}$ is principal.
 Recall that $A$ is said to be a \emph{B\'ezout domain} if every nonzero pair is a B\'ezout pair; equivalently, if every finitely-generated ideal is principal.

We will call an  element $x\in \projl{F}=F\cup\{ \infty\}$  a \emph{B\'ezout point}  if $x=a/b$ for some B\'ezout pair $(a,b)$.  (Note that if $(a,b)$ is  B\'ezout pair and if 
$a/b=a'/b'$ in $\projl{F}$, then $(a',b')$ is a B\'ezout pair. )

\begin{rem}  Of course, the notion of B\'ezout point or B\'ezout pair depends on the choice of subring $A$ of the field $F$.  If necessary, we will refer to an \emph{$A$-B\'ezout point} or 
an \emph{$A$-B\'ezout pair}.
\end{rem}
\begin{lem} \label{lem:bezout}
Let $A$ be an integral domain with field of fractions $F$. There is a natural injective map from $Y_0(A)$ to $\projl{F}$ sending $\bar{u}=(a,b)\to \frac{a}{b}$, whose
image is the set of B\'ezout points. 
\end{lem}
\begin{proof} The map $Y_0(A)\to \projl{F}$, sending the class of  $(a,b)$ to $ a/b$ is clearly  well-defined. Suppose that $(a,b),(c,d)$ are both unimodular and that $a/b=c/d$. 
Then $ad=bc$. So $a| bc$. Since there exist $r,s\in A$ with $ra+sb=1$, it follows that $a|c$. Similarly, $c|a$. Thus $c=ua$ for some $u\in A^\times$. It follows that 
$b=ud$ and hence $(a,b)=(c,d)$ in $Y_0(A)$. 

Suppose now that $(a,b)$ is a B\'ezout pair.  Let $q=a/b\in F$. Suppose that $\an{a,b}=\an{c}$ for $c\in A$. 
Then $a=a'c$ and $b=b'c$ for some $a',b',\in A$ with $(a',b')$ unimodular. So $a/b=a'/b'$ is the image of (the class of) $(a',b')\in Y_0(A)$.

Conversely, suppose that $a/b$ lies in the image of this map. Then $a/b=a'/b'$ for some $a',b'$ such that $(a',b')$ is unimodular. Since $a'|ab'$, it follows that $a'|a$. So $a=ca'$ for some $c\in A$.
We deduce that $b=cb'$ also. Hence $\an{a,b}=\an{ca',cb'}=\an{c}\an{a',b'}=\an{c}$ and $(a,b)$ is a B\'ezout pair.
\end{proof}
%\begin{rem}
 %\end{rem}

\begin{cor} 
If $A$ is a B\'ezout domain with field of fractions $F$ then the set of vertices, $Y_0(A)$, of $\Gamma(A)$ is naturally identified with $\projl{F}$.
\end{cor}

In general, when $A$ is an integral domain we will  identify  $Y_0(A)$ with the set of B\'ezout points in $\projl{F}$ .
\begin{rem}
Recall that $\gl{2}{A}$ acts transitively (on the right) on $Y_0(A)$ and this action is compatible with the natural right action on $\projl{F}$: If $M=\left[
\begin{array}{cc}
a&b\\
c&d\\
\end{array}
\right]\in \gl{2}{A}$, then $qM=\frac{aq+c}{bq+d}$ for all $q\in \projl{F}$.  Thus $Y_0(A)$, the set of B\'ezout points,  is precisely the orbit $\infty\cdot\gl{2}{A}\subset \projl{F}$. 
\end{rem}

\begin{exa} \label{exa:field}
If $F$ is a field then $\Gamma(F)$ is the complete graph on  $Y_0(F)=\projl{F}$. So  $Y_n(F)$ consists of all $(n+1)$-tuples $(x_0,\ldots,x_n)$ of distinct points of $\projl{F}$. 
If $F=\F{q}$ is the finite field with $q$ elements then $Y(F)$ is  just a $q$-simplex and hence the space $|Y(F)|$ is contractible. For infinite fields $F$, $Y(F)$ is an `infinite' simplex and it is also the case that $|Y(F)|$ is contractible. 
\end{exa}

%%%%%%%%%%%%%%%%%%%%%%%%%%%%%%%%%%%%%%%%%%%%%%%%%%%%%%%%%%%%%%%%%%%%%%%%%%%%%%%
\subsection{The edges  of $\Gamma(A)$}
%%%%%%%%%%%%%%%%%%%%%%%%%%%%%%%%%%%%%%%%%%%%%%%%%%%%%%%%%%%%%%%%%%%%%%%%%%
%Note that for $x\in A$, $\{ 0,x_+\}$ is an edge of $\Gamma(A)$ if and only if $x\in A^\times$ and $\{ 1,x_+\}$ is an edge if and only if  $1-x\in A^\times$.
\begin{lem}\label{lem:edge}
Let $A$ be a ring. Then $\spl{2}{A}$ and $\gl{2}{A}$ act transitively on the set $X_1(A)$ of ordered edges of $\Gamma(A)$ and the stabilizer of the ordered 
edge $( \infty, 0)$ is the group $T$ of diagonal matrices in
$\spl{2}{A}$ and $\gl{2}{A}$ respectively.
\end{lem}
\begin{proof}
Let $( [a], [b])$ be an ordered edge and let 
$
X:=\left[
\begin{array}{c}
a\\
b\\
\end{array}
\right]
\in \gl{2}{A}.
$  Replacing $a$ by $\mathrm{det}(X)^{-1}a$ if necessary, we can suppose $X\in \spl{2}{A}$. Clearly $( \infty\cdot X, 0\cdot X)= ( [a],[b])$. The statement about the stabilizer is clear.
\end{proof}

Let us say that two vertices $[u]$ and $[v]$ of $\Gamma(A)$ are \emph{neighbours} if $\{ [u],[v]\}$ is an edge. Then $[b_1,b_2]$ is a neighbour of $\infty=[1,0]$ if and only if 
$b_2\in A^\times$ and thus if and only if $[b_1,b_2]= [(b_1b_2^{-1},1)]$ lies in $A_+$. Likewise, the set of neighbours of $0$ is $A_-$. Hence the set of  common neighbours of $\infty$ and $0$ 
is the set $A_+\cap A_-=\{ [u,1]\ | u\in A^\times\}=\{ [1,u]\ | u\in A^\times\}$.

%%%%%%%%%%%%%%%%%%%%%%%%%%%%%%%%%%%%%%%%%%%%%%%%%%%%%%%%%%%%%%%%%%%
\subsection{$3$-cliques and $4$-cliques in $\Gamma(A)$}\label{sec:clique}
%%%%%%%%%%%%%%%%%%%%%%%%%%%%%%%%%%%%%%%%%%%%%%%%%%%%%%%%%
\begin{lem} \label{lem:3clique}
Let $A$ be a ring. $\gl{2}{A}$ acts transitively on the set of ordered $3$-cliques $X_3(A)$. The stabilizer of the ordered $3$-clique $(\infty,0,1)$ is the group of scalar matrices in
$\gl{2}{A}$.
\end{lem}
\begin{proof}
Let $([a],[b],[c])$ be any ordered $3$-clique. By the proof of Lemma \ref{lem:edge} we can choose $X\in \gl{2}{A}$ such that $[a]\cdot X=\infty, [b]\cdot X = 0$. It follows that 
$[c]\cdot X$ is a neighbour of both $\infty$ and $0$ and hence $[c]\cdot X = [u,1]$ for some $u\in A^\times$. Now let $X':= X\mathrm{diag}(1,u)$. Then 
$[a]\cdot X'=\infty, [b]\cdot X'=0$ and $[c]\cdot X'=[u,1]\cdot \mathrm{diag}(1,u)=[u,u]=1$ as required. 
\end{proof}

\begin{cor}\label{cor:3clique}
Let $a$, $b$ be unimodular rows and suppose that $\{ [a],[b]\}$ is an edge of $\Gamma(A)$. Then there is a bijection from $A^\times$ to the set of vertices $[c]$ such that 
$\{ [a], [b],[c]\}$ is a $3$-clique, given by $u\mapsto [ a+ub]$.
\end{cor}
\begin{proof} Let $X=\left[
\begin{array}{c}
a\\
b\\
\end{array}
\right]\in \gl{2}{A}$.
As remarked above, the $3$-cliques containing $\{ \infty,0\}$ are precisely the triples $\{\infty,0,[u,1]\}$ for some $u\in A^\times$. 
For any $u\in A^\times$,  $X$ sends the $3$-clique $(\infty, 0, [1,u])$ to $([a], [b],[a+ub])$ and hence this accounts for all $3$-cliques containing $[a]$ and $[b]$.
\end{proof}

\begin{lem}\label{lem:4clique} Let $A$ be a ring. The graph $\Gamma(A)$ contains $4$-cliques if and only if the set $\wn{A}:=\{ u\in A^\times\ |\ 1-u\in A^\times\}$ is non-empty.
\end{lem}
\begin{proof}
Let $\{ x,y,z,w\}$ be a $4$-clique in $\Gamma(A)$. There exists $X\in \spl{2}{A}$ such that $x=\infty\cdot X$ and $y=0\cdot X$. Multipying by $X^{-1}$ if necessary, we can 
assume $x=\infty$, $y=0$.  Since $z$ is a neighbour of $\infty$, $z\in A_+$. Since $z$ is a neighbour of $0$, $z\in A_-$. Thus $z\in A_+\cap A_-$ so $z=[r,1]$ for some unit $r$.
Similarly $w=[s,1]$ for some  unit $s$. Since $z$ is a neighbour of $w$, $r-s=t$ is a unit. Thus, letting $u=r^{-1}s\in a^\times$, we have $1-u=r^{-1}t\in A^\times$.

Conversely, if $u,1-u$ are units in $A$, then $\infty,0,[u,1],[1-u,1]\}$ is a $4$-clique in $\Gamma(A)$. 
\end{proof}

\begin{exa}
 Since $\wn{\Z}=\emptyset$, $\Gamma(\Z)$ has no $4$-cliques and hence $Y(\Z)$ has no $3$-simplices. Thus $Y(\Z)=Y_2(\Z)$  is $2$-dimensional simplicial complex. 
\end{exa}
%%%%%%%%%%%%%%%%%%%%%%%%%%%%%%%%%%%%%%%%%%%%%%%%%%%%%%%%%%%%%%%%%%%%%%%%%%%%
%\subsection{Some Examples}
%%%%%%%%%%%%%%%%%%%%%%%%%%%%%%%%%%%%%%%%%%%%%%%%%%%%%%%%%%%%%%%%%%%%%%%%%%%%%%%%

%%%%%%%%%%%%%%%%%%%%%%%%%%%%%%%%%%%%%%%%%%%%%%%%%%%%%%%%%%%%%%%%%%%%%%%
\subsection{Example: The graph $\Gamma(\Z)$ and  space $Y(\Z)$}\label{sec:gammaz}
%%%%%%%%%%%%%%%%%%%%%%%%%%%%%%%%%%%%%%%%%%%%%%%%%%%%%%%%%%%%%%%%%%%%%%%%
%\begin{exa}\label{exa:z}
Since $\Z$ is a PID, the vertices of $\Gamma(\Z)$ are naturally the points of $\projl{\Q}$. The neighbours of the vertex $\infty$ are precisely the integers $\Z\subset\projl{\Q}$ and each $n\in \Z$ is a neighbour of $n+1$. 
Recall that fractions $p/q$ and $r/s$ (written in reduced terms ) are neighbours in $\Gamma(\Z)$ if and only if  $\left| \frac{p}{q}-\frac{r}{s}\right|=\frac{1}{qs}$.
 Thus thus the structure of $\Gamma(\Z)$ 
and $Y(\Z)$ are closely related to the \emph{Farey sequences} $F_n$:

Recall that the \emph{$n$th Farey sequence} $F_n$ is the sequence of reduced fractions between (and  including) $0$ and $1$  with denominator at most $n$, written in ascending order of size. If $a/b,c/d$ are two successive terms in $F_n$ (called \emph{Farey neighbours}), then $\left| \frac{a}{b}-\frac{c}{d}\right|=\frac{1}{bd}$. Conversely, if the reduced fractions $a/b$ and $c/d$ lie between $0$ and $1$ and satisfy  $\left| \frac{a}{b}-\frac{c}{d}\right|=\frac{1}{bd}$, then $a/b$ and $c/d$ are neighbours in $F_n$ where $n=\mathrm{max}(b,d)$.
Observe also that $F_n\setminus F_{n-1}=\{ \frac{a}{n}\ |\ (a,n)=1\}$. Thus this set has $\phi(n)$ elements. 

We now embed the graph $\Gamma(\Z)$ and the  space $|Y(\Z)|$  in $\R^2$ as follows: Identify $\infty$ with the origin $(0,0)\in \R^2$. Now for any $x\in \Q$, write $x=\frac{p}{q}$ 
with $(p,q)=1$ and $q>0$ and identify $p/q$ with the point $(p/q,q)\in \R^2$.  We join any two points which are neighbours in $\Gamma(\Z)$ to get an embedding 
$\Gamma(\Z)\subset \R^2$. Our arguments below will show that none of the resulting edges intersect except at vertices. Recall also that each edge 
$(x,y)=(p/q,r/s)$   in $\Gamma(\Z)$ lies in precisely two $2$-simplices of $Y(\Z)$: those with third vertex $x\oplus y:=\frac{p+r}{q+s}$ or $x\ominus y:=\frac{p-r}{q-s}$ (see Corollary \ref{cor:3clique}).

We construct the subspace $\tilde{Y}(Z)\subset\R^2$ as follows:
We start with $\Gamma(\Z)$ and whenever $\{ x,y,z\}\subset \R^2$ is a $2$-simplex we add all the points in the closed triangle $[x,y,z]$ to $\tilde{Y}(Z)$. We will show that $\tilde{Y}(\Z)$ is homeomorphic to $|Y(\Z)|$.

For all $n\geq 1$, let $\Gamma(\Z,n)\subset \R^2$ be the induced subgraph of $\Gamma(\Z)$ consisting of vertices with denominator at most $n$ (including $\infty$, which has denominator $0$). Let 
$\tilde{Y}(\Z,n)$ be the corresponding subspace of $\tilde{Y}(\Z)$. Thus $\Gamma(\Z)=\cup_{n=1}^\infty \Gamma(\Z,n)$ and $\tilde{Y}(Z)=\cup_{n=1}^\infty \tilde{Y}(\Z,n)$.

Thus the vertices of $\Gamma(\Z,1)$  are $\infty$ and the integers. It decomposes into $2$-simplices $[\infty,n,n+1]$, $n\in \Z$,  and thus $\tilde{Y}(\Z,1)$ is the set
$\left(\R\times (0,1]\right)\cup \{ (0,0)\}$. 

Observe that for any $n\in \Z$, translation by $n$ ($=$ right multiplication by 
$\left[
\begin{array}{cc}
1&0\\
n&1\\
\end{array}
\right]
\in \spl{2}{A}$) is an automorphism of $\Gamma(\Z)\subset \projl{\Q}$. When we embed $\Gamma(\Z)$ in $\R^2$, it follows that $\Gamma(\Z)\cap H_1$ is invariant 
under horizontal translation by any $n\in \Z$, where $H_1$ is the closed half-plane $\{ (x,y)\in \R^2\ |\ y\geq 1\}$. Similarly, $\Gamma(\Z,m)\cap H_1$ and 
$\tilde{Y}(\Z,n)\cap H_1$  are invariant under horizontal translation by integers. In turn, it follows that the spaces $\tilde{Y}(\Z,n+1,n):= \tilde{Y}(\Z,n+1)\setminus \tilde{Y}(\Z,n)$,
$n\geq 1$ are all invariant under horizontal translations by integers. Now $\tilde{Y}(\Z)=\tilde{Y}(\Z,1)\cup \left(\bigcup_{n=1}^{\infty}\tilde{Y}(\Z,n+1,n)\right)$.

The vertices of $\Gamma(\Z,n)$ are $\infty$, the Farey sequence $F_n\subset \R^2$ and all horizontal translates of the points of $F_n$ by integer distances. 

We observe that if $x,y$ are rational with $x<0,y>0$ then $(x,y)$ cannot be an edge of $\Gamma(\Z)$: Let $x=-a/b,y=c,d$ with $0<a,b,c,d\in \Z$. Then the absolute value of the 
corresponding determinant is $|-ad-bc|=|ad+bc|>1$. By the integer translation invariance, it follows that if $x<n, y>n$ for any $n\in \Z$ then $(x,y)$ is not an edge of $\Gamma(\Z)$.
Thus each edge of $\Gamma(\Z)\cap H_1$ is fully contained in one of the vertical strips $[n,n+1]\times \R_+$.

The points and additional edges of of $\Gamma(\Z,n+1,n)$ arise as follows: Because of the translation-invariance property it is enough to determine the additional points and edges
in the vertical strip $[0,1]\times \R_+$.  The existing vertices are the points of $F_n$.  Whenever a pair of neighbours $x=p/q<r/s=y$ in this sequence has the property that $q+s=n+1$,
then $x\oplus y$ is a point of $F_{n+1}\setminus F_n\subset \Gamma(\Z,n+1,n)$ with $y$-coordinate $n+1$. We have $x<x\oplus y <y$, and we get two additional edges $(x,x\oplus y)$ and $(x\oplus y,y)$
 and an additional $2$-simplex $\{ x,x\oplus y,y\}$.
Recall that the edge $(x,y)$  bounds two $2$-simplices. The other has vertex $x\ominus y$ which lies strictly below  $x$ and $y$ since it has denominator $|q-s|<q,s$. 
In this way, the new edges  which are added in $\tilde{Y}(\Z,n+1,n)$ do not intersect the interior of any  any existing edges and the new simplices do not intersect the interior of any  any existing simplices.  

Thus we see by induction that each of the spaces  $|Y(\Z,n)|$ embeds $\R^2$ with image $\tilde{Y}(\Z,n)$. Hence $|Y(\Z)|$ embeds in $\R^2$ with image $\tilde{Y}(\Z)=\cup \tilde{Y}(\Z,n)$.

Note further that since $x<x\oplus y <y$, given any point $z$ in the closed simplex $\{ x,x\oplus y,y\}$ the vertical line segment joining $z$ to the edge $(x,y)$ is fully contained in this 
$2$-simplex. It follows that $\tilde{Y}(\Z,n+1)$ homotopy retracts to $\tilde{Y}(\Z,n)$ by moving along vertical lines, and hence  that the space $\tilde{Y}(\Z)$ homotopy retracts to the contractible space $\tilde{Y}(\Z,1)$  by restricting to $\tilde{Y}(\Z)$ the obvious contraction of $\R^2$ to $\{ (x,y)\ | y\leq 1\}$ along vertical lines. 
It follows that $\tilde{Y}(\Z)$, and hence $|Y(\Z)|$,  is a contractible space.

Finally, we describe explicitly the space $\tilde{Y}(\Z)\subset \R^2$.: 

Let $(x,t)\in \R^2$ with $t>0$.
\begin{enumerate}
\item If $x=\frac{p}{q}$ is  rational, then $(x,t)\in \tilde{Y}(\Z)$ if and only if $t\leq q$. Let $L_{p/q}$ be the open vertical half-line $ \{ (\frac{p}{q},t)\ | \ t>q\}$. Then 
$L_{p/q}\cap \tilde{Y}(\Z)=\emptyset$. 

\item Suppose now that $x$ is irrational.  Choose an integer $m>t$ . Choose also an integer $k>m$ satisfying $|x-w|<1/k$ for all $w\in F_m$. Now let $y=p/q$ be the largest 
element of $F_k$ which is smaller than $x$ and let $z=r/s$ be the successor of $y$ in $F_k$. So we have $y<x<z$. Furthermore $|y-x|<1/k$ and thus $y\not\in F_m$: i.e., $q>m>t$.
Similarly, $s>m>t$. 
% Choose rationals $y=p/q$ and $z=r/s$ satisfying (i) $q,s>t$ (ii) $y$ and $z$ are neighbours in the Farey sequence $F_n$ where
% $n=\mathrm{max}(q,s)$ and (iii) $y<x<z$.
 Thus  $(x,t)$ lies vertically below the edge $((y,q),(z,s))\in \tilde{Y}(\Z,n)$ and hence $(x,t)$ itself lies in 
$ \tilde{Y}(\Z,n)\subset \tilde{Y}(\Z)$.
\end{enumerate}
Putting these two observations together, we conclude that 
\[
\tilde{Y}(\Z)=\{ (0,0)\}\cup \left( \R\times (0,\infty)\right)\setminus\left(\cup_{p/q} L_{p/q}\right).
\]
%\end{exa}

%%%%%%%%%%%%%%%%%%%%%%%%%%%%%%%%%%%%%%%%%%%%%%%%%%%%%%%%%%%%%%%%%%%%%%
\section{Path components of $\Gamma(A)$ and the condition ($\mathrm{GE}_2$)}\label{sec:path}
%%%%%%%%%%%%%%%%%%%%%%%%%%%%%%%%%%%%%%%%%%%%%%%%%%%%%%%%%%%%%%%%%%%%%%%%%%%%%%
Let $A$ be a ring. Given $a\in A$, we consider the following matrices in $\spl{2}{A}$:
\[
E_{2,1}(a):=\left[
\begin{array}{ll}
1&0\\
a&1\\
\end{array}
\right],\  E_{1,2}(a):=
\left[
\begin{array}{ll}
1&a\\
0&1
\end{array}
\right],\ 
E(a):=
\left[
\begin{array}{ll}
a&1\\
-1&0
\end{array}
\right]=W\cdot E_{2,1}(a)= E_{1,2}(-a)\cdot W
\]
where 
\[
W:=E(0)=
\left[
\begin{array}{ll}
0&1\\
-1&0
\end{array}
\right]=
E_{1,2}(1)E_{2,1}(-1).
\]

Let $E_2(A)$ denote the subgroup of $\spl{2}{A}$ generated by the set of elementary matrices $E_{i,j}(a)$, $a\in A$ for $i\not= j\in \{ 1,2\}$. Let $\mathrm{GE}_2(A)$ 
denote the subgroup of $\gl{2}{A}$ generated by  $E_2(A)$ and the group $T=D_2(A)$ of invertible  diagonal matrices. Thus all the above matrices belong to $\mathrm{E}_2(A)$.
Furthermore, for any $a\in A$  we let
\[
S(a):=\left[
\begin{array}{ll}
a&1\\
1&0
\end{array}
\right]=
\left[
\begin{array}{ll}
1&0\\
0&-1
\end{array}
\right]\cdot E(a)\in \gl{2}{A}.
\]

\begin{lem}\label{lem:path} Let $A$ be a ring and let $x_0,x_1,\ldots,x_n$ be a path in $\Gamma(A)$. Fix $X\in \gl{2}{A}$ with $x_0=\infty\cdot X$. Then there exist unique
$a_1,\ldots,a_n\in A$ satisfying
 \[
x_i=\infty E(a_i)E(a_{i-1})\cdots E(a_1)X\mbox{ for }i=1,\ldots, n.
\]
Conversely, given $X\in\gl{2}{A}$ and $a_1,\ldots, a_n\in A$ , if we set $x_0:=\infty \cdot X$ and $x_i:=\infty\cdot E(a_i)\cdots E(a_1)X$ for $1\leq i\leq n$, then
$x_0,x_1,\cdots,x_n$ is a path in $\Gamma(A)$. 
\end{lem}
\begin{proof} For the first statement, we proceed by induction on $n$. When $n=1$, we have that $x_1$ is a neighbour of $x_0=\infty\cdot X$ and hence $x_1\cdot X^{-1}$ is a neighbour of
 $x_0\cdot X^{-1}=\infty$. 
It follows that $x_1\cdot X^{-1}=[a_1,1]=\infty\cdot E(a_1)$ for some uniquely determined $a_1\in A$. Thus $x_1=\infty\cdot E(a_1)X$ as required. For the inductive step, simply 
replace $X$ by $X_n:= E(a_n)\cdots E(a_1)X$.

The converse statement follows from the observation that if $Y\in \gl{2}{A}$ and $a\in A$, then $\{ \infty, \infty\cdot E(a)\}$ is an edge in $\Gamma(A)$, and hence so is
$\{ \infty\cdot Y,\infty\cdot E(a)Y\}$. 
\end{proof}

\begin{rem} Since for any $a\in A$ we have $[a,1]=\infty\cdot E(a)=\infty\cdot S(a)$, we could have stated Lemma \ref{lem:path} using the matrices $S(a)$ rather than the 
matrices $E(a)$. One then finds $x_i=\infty E(a_n)\cdots E(a_1)X= \infty S(b_n)\cdots S(b_1)X$ where $b_i=(-1)^{i-1}a_i$ for all $i$.
\end{rem}

\begin{thm}\label{thm:pi0}  Let $A$ be a ring. Then there is  a natural bijection of right $\gl{2}{A}$-sets $\pi_0(\Gamma(A))\leftrightarrow \mathrm{GE}_2(A)\backslash \gl{2}{A}$.
\end{thm}

\begin{proof} For $u,v\in A^\times$, let $D=\mathrm{diag}(u,v), D'=\mathrm{diag}(v,u)\in \gl{2}{A}$. Then for any $a\in A$, we have $E(a)D=D'E(v^{-1}au)$. It follows that any element of $\mathrm{GE}_2(A)$ can be written as a product $DE$ where $D$ is a diagonal matrix and $E\in E_2(A)$.

For $x\in Y_0(A)$, let $p(x)$ denote the path component of $x$ in $\Gamma(A)$. Suppose now that vertices $x$ and $y$ are in the same path component of $\Gamma(A)$. 
Choose $X,Y\in\gl{2}{A}$ such that $x=\infty\cdot X$ and $y=\infty\cdot Y$. By Lemma \ref{lem:path} there exist $a_1,\ldots,a_n\in A$ such that
\[
y=\infty\cdot  E(a_n)\cdots E(a_1)X=\infty\cdot Y.
\]
It follows that $E(a_n)\cdots E(a_1)XY^{-1}$ stabilizes $\infty$ and hence belongs to $\tilde{B}\subset \mathrm{GE}_2(A)$. Thus $\mathrm{GE}_2(A)X=\mathrm{GE}_2(A)Y$ in 
$\mathrm{GE}_2(A)\backslash \gl{2}{A}$. We have shown that there is a well-defined map of right $\gl{2}{A}$-sets 
\[
\pi_0(\Gamma(A))\to \mathrm{GE}_2(A)\backslash \gl{2}{A},\ 
p(x)\mapsto \mathrm{GE}_2(A) X
\]
where $X\in\gl{2}{A}$ is any matrix such that $x=\infty\cdot X$. 

Conversely,  the map
 \[
\gl{2}{A}\to \pi_0(\Gamma(A),\ X\mapsto p(\infty\cdot X)
\]
 induces a well-defined map $\mathrm{GE}_2(A)\backslash \gl{2}{A}\to \pi_0(A)$: Let $Y\in \mathrm{GE}_2(A)$. By our preliminary remarks, $Y=DE$ where $D$ is diagonal and $E\in E_2(A)$.
Thus $\infty\cdot YX=\infty \cdot DEX=\infty \cdot EX$ lies in $p(x\cdot X)$ by Lemma \ref{lem:path} again. This map is now clearly inverse to the map defined in the previous paragraph.
\end{proof}

\begin{rem}
For any ring $A$, the inclusion $\spl{2}{A}\to \gl{2}{A}$ induces a bijection of right $\spl{2}{A}$-sets $E_2(A)\backslash \spl{2}{A}\leftrightarrow \mathrm{GE}_2(A)\backslash \gl{2}{A}$.
\end{rem}

We recall  that a ring $A$ \emph{satisfies condition ($\mathrm{GE}_2$)}, or \emph{is a $\mathrm{GE}_2$-ring}, if $\mathrm{GE}_2(A)=\gl{2}{A}$ (Cohn, \cite{cohn:gln}). Since 
$E_2(A)=\spl{2}{A}\cap \mathrm{GE}_2(A)$ and $\gl{2}{A}=\spl{2}{A}\cdot D_2(A)$, this condition is equivalent to $E_2(A)=\spl{2}{A}$.

\begin{cor} \label{cor:pi0}
Let $A$ be a ring. Then $A$ is a $\mathrm{GE}_2$-ring if and only if the graph $\Gamma(A)$ is path-connected. 
\end{cor}

We recall how these properties are related to the classical Euclidean algorithm:

\begin{lem}\label{lem:euclid}
 Let $A$  be a ring and let $(a,b)$ be a unimodular row. The following are equivalent:
\begin{enumerate}
\item There exist $a_0,\ldots,a_n\in A$ such that $[a,b]= \infty\cdot S(a_n)S(a_{n-1})\cdots S(a_0)$.
\item There exist $a_0,\ldots, a_n, ,r_0,\ldots, r_{n-1}\in A$ satisfying the following:\\
 Let $r_{-2}:=a,r_{-1}:=b,r_n:=0$. Then 
\[
r_{k-2}=a_kr_{k-1}+r_k \mbox{ for } 0\leq k\leq n.
\]
(Following the terminology of \cite{czz}, we will say that the pair $(a,b)$  \emph{satisfies a weak Euclidean algorithm}.)
\end{enumerate}
\end{lem}

\begin{proof} We begin by noting that for $a\in A$ and $(x,y)\in A^2$, we have 
\[
(x,y)S(a)^{-1}=(x,y)\left[
\begin{array}{ll}
0&1\\
1&-a
\end{array}
\right]=(y,x-ay).
\]
Thus $(z,w)=(x,y)S(a)^{-1}$ if and only if $z=y$ and $x=ay+w$. 

Suppose now that (1) holds. Lifting the equation to $A^2$ we have 
\[
(a,b)=(u,0)\cdot S(a_n)S(a_{n-1})\cdots S(a_0)
\]
for some unit $u$  in $A$. Now define $(r_{-2},r_{-1}):=(a,b)$ and $(r_{k-1},r_k):=(a,b)\cdot S(a_0)^{-1}\cdots S(a_k)^{-1}$ for $0\leq k\leq n$.
These definitions are consistent since they satisfy $(r_{k-1},r_k)=(r_{k-2},r_{k-1})S(a_k)^{-1}$. Furthermore it follows that $r_{k-2}=a_kr_{k-1}+r_k$.  Finally note that
$(r_{n-1},r_n)=(a,b)\cdot S(a_0)^{-1}\cdots S(a_n)^{-1}=(u,0)$ so that $r_n=0$ (and $r_{n-1}=u$), as required.

Conversely, suppose that (2) holds. Then we have that $(r_{-2},r_{-1})=(a,b)$ and $(r_{k-1},r_k)=(r_{k-2},r_{k-1})S_{a_k}^{-1}$ for $0\leq k\leq n$. It follows that 
\[
(a,b)\cdot S(a_0)^{-1}\cdots S(a_n)^{-1}=(r_{n-1},r_n)=(r_{n-1},0)
\]
where $r_{n-1}$ is necessarily a unit. Hence (1) holds (since then $\infty =[r_{n-1},0]$). 
\end{proof}

\begin{prop}\label{prop:euclid} For any ring $A$, the following are equivalent:
\begin{enumerate}
\item $\Gamma(A)$ is path-connected.
\item For any unimodular row $(a,b)$,  there exist $a_0,\ldots,a_n\in A$ such that \\
$(a,b)= \infty\cdot S(a_n)S(a_{n-1})\cdots S(a_0)$.
\item Any unimodular pair $(a,b)$ satisfies a weak Euclidean algorithm. 
\item $A$ is a $\mathrm{GE}_2$-ring.
\end{enumerate}
\end{prop}
\begin{proof} 
(1) and (4) are equivalent by Corollary \ref{cor:pi0}. (1) and (2) are equivalent by Lemma \ref{lem:path} (and the remark that follows it). (2) and (3) are equivalent by Lemma \ref{lem:euclid}.
\end{proof}

\begin{exa}  Dennis, Magurn and Vaserstein (\cite{dmv},1984) have shown that $\Z[C_n]$ is $GE_2$ring for all $n\geq 1$, where $C_n$ denotes the cyclic group of order $n$. 
Thus the graph $\Gamma(\Z[C_n])$ is path-connected.
\end{exa}

\begin{exa} 
 If $\mathcal{O}$ is a ring of $S$-integers in a number field $F$ and if the group of units $\mathcal{O}^\times$  is infinite then $\mathcal{O}$ is a $\mathrm{GE}_2$-ring (Vaserstein \cite{vas},1972), and hence $\Gamma(\mathcal{O})$ is path-connected. 
\end{exa}
%%%%%%%%%%%%%%%%%%%%%%%%%%%%%%%%%%%%%%%%%%%%%%%%%%%%%%%%%%%%%%%%%%%%%%%%%%%%%%%%%%%%
%\section{Integral Domains}
%%%%%%%%%%%%%%%%%%%%%%%%%%%%%%%%%%%%%%%%%%%%%%%%%%%%%%%%%%%%%%%%%%%%%%%%%%%%%%%%%

When $A$ is an integral domain the conditions of Proposition \ref{prop:euclid} are also equivalent to the existence of certain continued fraction expansions:

\begin{rem}
Since $x\cdot S(a)=a+\frac{1}{x}$ for $x\in \projl{F}$, we have $\frac{a}{b}=\infty\cdot S(a_n)S(a_{n-1})\cdots S(a_0)$ if and only if 
\[
\frac{a}{b}=
a_0+\cfrac{1}{a_{1}+\cfrac{1}{a_{2}+\cfrac{1}{a_{3}+\cfrac{1}{\ddots+\cfrac{1}{a_n}}}}}\quad\mbox{ in }\projl{F}.
\]
Thus the B\'ezout point $a/b$ lies in the path component of $\infty\in \Gamma(A)$ if and only if $a/b$ admits a finite continued fraction expansion  (with entries in $A$).

Of course, it also follows that any point in $\projl{F}$ admitting a finite continued fraction expansion is necessarily a B\'ezout point, since it lies in the orbit $\infty\cdot \gl{2}{A}$. 
(The connection between weak Euclidean algorithms and continued fraction expansions goes back at least to the work of G.E.Cooke, \cite{cooke:ge}.) 
\end{rem}

We thus have:
\begin{cor}\label{cor:pathconn} Let $A$ be an integral domain with field of fractions $F$. The following are equivalent:
\begin{enumerate}
\item $\Gamma(A)$ is path-connected.
\item For every B\'ezout point $x\in \projl{F}$ there exist $a_0,\ldots,a_n\in F$ with\\
 $x=\infty\cdot S(a_n)S(a_{n-1})\cdots S(a_0)$.
 \item Every B\'ezout pair satisfies a weak Euclidean algorithm.
\item $A$ is  a $\mathrm{GE}_2$-ring. 
\item Every B\'ezout point in $\projl{F}$ admits a finite continued fraction expansion  with entries in $A$.
\end{enumerate}
\end{cor}

\begin{cor}\label{cor:euclid}
 If $A$ is a Euclidean domain then $Y_0(A)=\projl{A}$ and $\Gamma(A)$ is path-connected. 
\end{cor}
\begin{proof} 
A Euclidean domain is a principal ideal domain, hence a B\'ezout domain. Any non-zero pair of elements of $A$ satisfies a weak Euclidean algorithm using the division algorithm  in the usual way. 
\end{proof}

On the other hand, there exist PIDs $A$ for which $\Gamma(A)$ is not path-connected:
\begin{exa}[ Cohn, \cite{cohn:gln}] \label{exa:cohn} 
The ring $A=\Z\left[ \frac{1+\sqrt{-19}}{2}\right]$ is a principal ideal domain but is not a $\mathrm{GE}_2$-ring. Thus 
$\Gamma(A)$ is not path-connected. 
\end{exa}

\begin{exa}[Cossu, Zanardo, Zannier, \cite{czz}]\label{exa:czz} 
 Let $k$ be a field. Let $\mathcal{C}\subset \mathbb{P}^2$ be a smooth projective curve of genus $0$ over $k$. Let $\mathcal{C}_0:=\mathcal{C}\cap \mathbb{A}^2$. 
If $\mathcal{C}(k)=\emptyset$ then the coordinate ring $k[\mathcal{C}_0]$ is a PID which is not a $\mathrm{GE}_2$-ring (\cite[Corollary 3.6]{czz}). 

 For example, take 
$k=\R$ and let $\mathcal{C}$ be the curve with homogeneous equation $x^2+y^2+z^2=0$. Then the ring $\R[\mathcal{C}_0]\cong \R[x,y]/\an{ x^2+y^2+1}$ is a PID but 
not a $\mathrm{GE}_2$-ring. 
\end{exa}

\begin{exa}[Cossu, Zanardo, Zannier, \cite{czz} again]\label{exa:czz2} 
 Let $k$ be a field. Let $\mathcal{C}$ be a smooth curve of genus $\geq 1$ over $k$ with a unique point at infinity. If the coordinate ring $k[\mathcal{C}_0]$ is a PID then it is not 
a $\mathrm{GE}_2$-ring (\cite[Theorem 3.11, Corollary 3.12]{czz}).

For example, let $k$ be a perfect field. Let $f(x)\in k[x]$ be a cubic with three distinct roots. Let $E\subset \mathbb{P}^2$ be th elliptice curve $y^2=f(x)$. It has a unique point, $P_\infty$ say,  at infinity. Suppose that $E(k)=\{ P_\infty\}$. Then the affine coordinate ring $k[E_0]=k[x,y]\an{y^2-f(x)}$ is a PID which is not a $\mathrm{GE}_2$-ring.  
\end{exa}

%%%%%%%%%%%%%%%%%%%%%%%%%%%%%%%%%%%%%%%%%%%%%%%%%%%%%%%%%%%%%%%
\section{Rings which are universal for $\mathrm{GE}_2$ and the group $C(A)$}\label{sec:univ}
%%%%%%%%%%%%%%%%%%%%%%%%%%%%%%%%%%%%%%%%%%%%%%%%%%%%%%%%%%%%%%%%%%%
Cohn \cite{cohn:gln} notes some universal relations satisfied by elementary matrices: Let $A$ be a ring.  For $u\in A^\times$, let $D(u):=\mathrm{diag}(u,u^{-1})\in \spl{2}{A}$. Observe that 
$D(u)=E(-u)E(-u^{-1})E(-u)\in E_2(A)$. Then we have
\begin{eqnarray*}
D(u)D(v)&=& D(uv) \mbox{ for all } u,v\in A^\times. \\
E(a)E(0)E(b)&=&-E(a+b)=D(-1)E(a+b)\mbox{ for all }a,b\in A\\
D(u)E(a)D(u)&=& E(u^2a)\mbox{ for all } a\in A,u\in A^\times.\\
\end{eqnarray*}
A ring $A$ a said to be a \emph{universal for $\mathrm{GE}_2$} if all relations among elementary matrices in $E_2(A)=\spl{2}{A}$ are consequences of these three
(families of) relations. We now spell out this condition in more explicitly.

Let $C(A)$ denote the group with generators $\epsilon(a),a\in A$ subject to the following three relations:
\begin{enumerate} 
\item For $u\in A^\times$, let $h(u):=\epsilon(-u)\epsilon(-u^{-1})\epsilon(-u)$. Then $h(u)h(v)=h(uv)$ for all $u,v\in A^\times$.
\item $\epsilon(a)\epsilon(0)\epsilon(b)=h(-1)\epsilon(a+b)$ for all $a,b\in A$.
\item $h(u)\epsilon(a)h(u)=\epsilon(u^2a)$ for all $u\in A^\times$, $a\in A$. 
\end{enumerate}
Because of Cohn's relations above, there is a well-defined  group homomorphism $\psi:C(A)\to \spl{2}{A}$ sending $\epsilon(a)$ to $E(a)$ (and hence $h(u)$ to $D(u)$). The image of $\psi$ is clearly 
$E_2(A)$ and thus $A$ is a $\mathrm{GE}_2$-ring if and only if $\psi$ is surjective. Furthermore,  $A$ is universal for $\mathrm{GE}_2$ if and only if 
$\psi$ is injective.  We will denote $\ker{\psi}$ by $U(A)$. So a ring $A$ is  universal for $\mathrm{GE}_2$ if and only if $U(A)=\{ 1\}$. 

Relation (1) implies $h(1)=1$ in $C(A)$ and hence $h(-1)^2=h(1)=1$. Furthermore $h(-1)$ is central in $C(A)$ by (3). We note also that letting $a=b=0$ in (2) gives the identity
$\epsilon(0)^2=h(-1)$ and hence $\epsilon(0)^4=1$. 

For any $a\in A$, letting $b=-a$ in (2) gives $\epsilon(a)\epsilon(0)\epsilon(-a)=h(-1)\epsilon(0)=\epsilon(0)^3=\epsilon(0)^{-1}$ and hence 
\[
\epsilon(a)^{-1}=\epsilon(0)\epsilon(-a)\epsilon(0)\mbox{ in } C(A)
\]
for all $a\in A$. 

For $a\in A$, let $y(a):=\epsilon(0)^3\epsilon(a)=h(-1)\epsilon(0)\epsilon(a)\in C(A)$. Thus $\psi(y(a))=E(0)^3E(a)=E_{2,1}(a)\in B:=B(\spl{2}{A})$. 
\begin{lem}\label{lem:ya} Let $A$  be a ring.
\begin{enumerate}
\item $y(a)y(b)=y(a+b)$ in $C(A)$  for all $a,b\in A$.
\item  For all $u\in A^\times$, $a\in A$ we have 
\[
y(a)^{h(u)}:=h(u)^{-1}y(a)h(u)=y(u^2a) \mbox{ in }C(A).
\]
\end{enumerate}
\end{lem}
\begin{proof}
\begin{enumerate}
\item Let $a,b\in A$. Then
\begin{eqnarray*}
y(a)y(b)&=& h(-1)\epsilon(0)\epsilon(a)h(-1)\epsilon(0)\epsilon(b)\\
&=& \epsilon(0)\epsilon(a)\epsilon(0)\epsilon(b)\mbox{ (since $h(-1)$ is central, $h(-1)^2=1$)}\\
&=& \epsilon(0)h(-1)\epsilon(a+b)=y(a+b)\mbox{ (using relation (2))}.
\end{eqnarray*}
\item Let $u\in A^\times$, $a\in A$. Then 
\begin{eqnarray*}
h(u)^{-1}y(a)h(u)&=& h(u^{-1})h(-1)\epsilon(0)\epsilon(a)h(u)\\
&=&h(-1)h(u^{-1})\epsilon(0)h(u^{-1})h(u)\epsilon(a)h(u) \mbox{( using relation  (1))}\\
&=& h(-1)\epsilon(0)\epsilon(u^2a)=y(u^2a) \mbox{ (using relation (2))}.
\end{eqnarray*}
\end{enumerate}
\end{proof}
 More generally, for any $u\in A^\times$, $a\in A$ we let
$\beta(u,a):=h(u)y(ua)\in C(A)$. Thus
\[
\psi(\beta(u,a))=D(u)E_{2,1}(ua)=
\left[
\begin{array}{cc}
u&0\\
a&u^{-1}\\
\end{array}
\right]\in B.
\]

\begin{lem}\label{lem:b}
Let $\mathbb{B}$ denote the subset $\{ \beta(u,a)\in C(A)\ |\ u\in A^\times, a\in A\}$ of $C(A)$. Then $\mathbb{B}$ is a subgroup and $\psi:\mathbb{B}\to B$ is a group isomorphism.
\end{lem}

\begin{proof} Let $u,v\in A^\times$, $a,b\in A$. Then
\begin{eqnarray*}
\beta(u,a)\beta(v,b)&=&h(u)y(ua)h(v)y(vb)\\
&=&h(u)h(v)y(uav^2)y(vb) \mbox{ by Lemma \ref{lem:ya} (2)}\\
&=& h(uv)y(uav^2+vb)\mbox{ by Lemma \ref{lem:ya} (1) }\\
&=& h(uv)y(uv(av+u^{-1}b))\\
&=&\beta(uv,av+u^{-1}b).\\
\end{eqnarray*}
It follows that $\mathbb{B}$ is closed under multiplication in $C(A)$. Furthermore this formula tells us that $\beta(u,a)^{-1}=\beta(u^{-1},-a)\in \mathbb{B}$ for all 
$u\in A^\times$, $a\in A$. So $\mathbb{B}$ is a subgroup of $C(A)$ and $\psi:\mathbb{B}\to B$ is a surjective homomorphism. 

Of course $\psi(\beta(u,a))=D(u)E_{2,1}(ua)=1$ if and only if $u=1$ and $a=0$. Thus $\ker{\psi}\cap \mathbb{B}=\{ \beta(1,0)\}=\{ 1\}$. So $\psi$ is an isomorphism 
as required. 
\end{proof}

We will regard the inverse isomorphism $B\to \mathbb{B}$ and the resulting embedding $B\to C(A)$ as the standard embedding of $B$ in $C(A)$ and will denote this map by $\mathrm{st}$.
Thus, in particular, $\mathrm{st}(E_{2,1}(a))=y(a)$ for all $a\in A$ and $\mathrm{st}(D(u))=h(u)$ for all $u\in A^\times$.

\begin{cor}\label{cor:b} The map $\mathbb{B}\to A^\times$, $\beta(u,a)\mapsto u$ is a well-defined group homomorphism whose kernel is the subgroup $\{ y(a)\ |a\in A\}$ which is 
isomorphic to the additive group $A$.
\end{cor}
If $\beta\in \mathbb{B}$, we will let $u(\beta)$ denote the associated unit. 

Recall that $\spl{2}{A}$ acts on $Y_0(A)$ and hence $B$ acts by restriction. Furthermore, if $a=[a,1]\in A_+$, then $a\cdot X\in A$ for all $X\in B$ (since $B$ stabilizes $\infty$ 
and sends neighbours to neighbours).
 Via the isomorphism $\psi$ there is thus a natural action of $\mathbb{B}$  on $A$; i.e., $a\cdot \beta:= a\cdot\psi(\beta)$ for all $\beta\in \mathbb{B}$ (and hence also, 
$a\cdot Z=a\cdot \stand{Z}$ for all $Z\in B$).  Explicitly we have:
\[
a\cdot \beta(u,b)=au^2+bu\mbox{ for all }a,b\in A, u\in A^\times.
\]

The following lemma will play a key role below.

\begin{lem}\label{lem:key}
Let $a\in A$, $\beta\in \mathbb{B}$. Then
\[
\epsilon(a\cdot \beta)=h(u(\beta))\epsilon(a)\beta \mbox{ in } C(A).
\]
\end{lem}

\begin{proof} Let $u=u(\beta)$. So $\beta=\beta(u,b)=h(u)y(ub)$ for some $b\in A$. Thus
\begin{eqnarray*}
\epsilon(a)\beta=\epsilon(a)h(u)y(ub)\\
&=& h(u^{-1})\epsilon(u^2a)y(ub) \mbox{ by (3)}\\
&=& h(u^{-1})\epsilon(u^2a)h(-1)\epsilon(0)\epsilon(ub)\\
&=& h(u^{-1})\epsilon(u^2a+ub) \mbox{ by (2)}\\
&=& h(u)^{-1}\epsilon(a\cdot\beta)
\end{eqnarray*}
as required.
\end{proof}

Applying the map $\psi$ (or, by direct calculation) we have
\begin{cor}\label{cor:key1}
Let $a\in A$, $X\in B $. Then $a\cdot X\in A$ and 
\[
E(a\cdot X)=D(u(X))E(a)X\mbox{ in } E_2(A),
\]
where the unit $u(X)$ is the $(1,1)$-entry of $X$. 
\end{cor}

\begin{prop}\label{prop:key}
Let $a_1,\ldots,a_n\in A$. Let $\beta\in \mathbb{B}\subset C(A)$ and let $u=u(\beta)$. Then there exist unique $b_1,\ldots,b_n\in A$ satisfying 
\[
\epsilon(b_i)\cdots\epsilon(b_1)=h(u)^{(-1)^{i-1}}\epsilon(a_i)\cdots\epsilon(a_1)\beta \mbox{ in }C(A) \mbox{ for all }i.
\]
Furthermore, we have 
\[
b_1=a_1\cdot \beta\mbox{ and }b_i=u^{(-1)^{i-1}2}a_i\mbox{ for all }i\geq 2.
\]
\end{prop}

\begin{proof} The uniqueness part of the statement follows from the observation that $\epsilon(b)=\epsilon(b')$ in $C(A)$ implies $b=b'$ in $A$. (Apply the homomorphism $\psi$ 
and this follows from the corresponding statement for the elements $E(b)$ in $\spl{2}{A}$.) Thus it is enough to verify that the elements given in the final line of the statement of the theorem satisfy the relevant identities in $C(A)$.

 We will proceed by induction on $n$. The case $n=1$ is just Lemma \ref{lem:key}. Suppose the result is known for a given $n\geq 1$ and that $a_1,\ldots,a_n,a_{n+1}\in A$ are given.
Let $b_1,\ldots,b_{n+1}$ be the elements defined in the final line of the theorem. Then
\begin{eqnarray*}
\epsilon(b_{n+1})\epsilon(b_n)\cdots \epsilon(b_n)&=&\epsilon(u^{(-1)^n2}a_{n+1})h(u)^{(-1)^{n-1}}\epsilon(a_n)\cdots\epsilon(a_1)\beta\\
&=& h(u)^{(-1)^n}\left(h(u)^{(-1)^{n-1}}\epsilon(u^{(-1)^n2}a_{n+1})h(u)^{(-1)^{n-1}}\right)\epsilon(a_n)\cdots\epsilon(a_1)\beta\\
&=& h(u)^{(-1)^n}\epsilon(a_{n+1})\epsilon(a_n)\cdots\epsilon(a_1)\beta\\
\end{eqnarray*}
as required, using defining relation (3) of $C(A)$. 
\end{proof}

Applying the homomorphism $\psi$ we deduce the corresponding statement for the group $\spl{2}{A}$:
\begin{cor}\label{cor:key2}
Let $a_1,\ldots,a_n\in A$. Let $Z\in B\subset \spl{2}{A}$ and let $u=u(Z)$. Then there exist unique $b_1,\ldots,b_n\in A$ satisfying 
\[
E(b_i)\cdots E(b_1)=D(u)^{(-1)^{i-1}}E(a_i)\cdots E(a_1)Z \mbox{ in }\spl{2}{A}\mbox{ for all }i.
\]
Furthermore, we have 
\[
b_1=a_1\cdot Z\mbox{ and }b_i=u^{(-1)^{i-1}2}a_i\mbox{ for all }i\geq 2.
\]
\end{cor}
%%%%%%%%%%%%%%%%%%%%%%%%%%%%%%%%%%%%%%%%%%%%%%%%%%%%%%%%%%%%%
\section{The Edge-Path Groupoid of the simplicial complex $Y(A)$}
%%%%%%%%%%%%%%%%%%%%%%%%%%%%%%%%%%%%%%%%%%%%%%%%%%%%%%%%
\subsection{The edge-path groupoid}
We begin by reviewing the notion of the \emph{edge-path groupoid} of a simplicial complex $Y$  (Spanier \cite[Chapter 4]{spanier}. A path $\mathbf{p}$ in $Y$ is an $(n+1)$-tuple $(x_0,x_1,\ldots, x_n)$ for some 
$n\geq 0$, where the $x_i$ are vertices (or $0$-simplices) and for $i=0,n-1$, each  $\{ x_i,x_{i+1}\}$ is and edge (or $1$-simplex) of $Y$.   If $\mathbf{p}=(x_0,\ldots, x_n)$ is a path, we 
set $i(\mathbf{p}):=x_0$, the initial point of the path,  and $t(\mathbf{p})=x_n$, the terminal point. 

If $\mathbf{p}=(x_o,\ldots,x_n)$ and $\mathbf{q}=(y_0,\ldots,y_m)$ are two paths satisfying $t(\mathbf{p})=i(\mathbf{q})$ (i.e., $x_n=y_0$) then we can form the \emph{concatenation}
$\mathbf{p}\star\mathbf{q}:=(x_0,\ldots,x_n,y_1,\ldots, y_m)$.

 We consider the equivalence relation generated by the following two relations (which we will refer to as the \emph{homotopy relations}):
\begin{enumerate}
\item $(x_0,\ldots, x_i,x_{i+1},x_{i+2},\ldots,x_n) \sim (x_0,\ldots,x_i,\ldots,x_n)$ for $0\leq i\leq n-2$ whenever $x_i=x_{i+2}$
\item $(x_0,\ldots, x_i,x_{i+1},x_{i+2},\ldots,x_n)\sim (x_0,\ldots,x_i,x_{i+2},\ldots, x_n)$ for $0\leq i\leq n-2$ whenever $\{ x_i,x_{i+1},x_{i+2}\}$ is a $2$-simplex of $Y$.
\end{enumerate}
We denote the equivalence path of $\mathbf{p}=(x_0,\ldots,x_n)$ by $[\mathbf{p}]=[x_0,\ldots,x_n]$. Note that this equivalence relation preserves initial and terminal points of paths and that
 concatenation of equivalence classes is well-defined by $[\mathbf{p}]\star[\mathbf{q}]=[\mathbf{p}\star\mathbf{q}]$ is $t(\mathbf{p})=i(\mathbf{q})$.

\subsection{Algebraic representation of paths and concatenation}
Let now $A$ be ring. We have already noted, in Lemma \ref{lem:path} above, that a path $(x_0,\ldots, x_n)$ 
 in $\Gamma(A)$ (or, equivalently, in $Y(A)$) is determined by an element $X\in \spl{2}{A}$ and a sequence $a_1,\ldots, a_n$ of $A$, uniquely determined by $\mathbf{p}$ and the choice of $X$ such that $\infty\cdot X=x_0$. In this case,  we have $x_i=\infty \cdot E(a_i)\cdots E(a_1)X$ for each $i\geq 1$. 

Since the stabilizer of $\infty$ in $\spl{2}{A}$ is the group $B$ of lower-triangular matrices, $X$ is determined up to left multiplication by an element of $B$.
 Corollary \ref{cor:key2} allows us to  track the effect that changing the choice of $X$ has on the sequence of elements $a_1,\ldots, a_n$. It immediately implies:

\begin{prop}\label{prop:changeX} Let $\mathbf{p}=(x_0,\ldots,x_n)$  be a path in $Y(A)$. Let $X\in \spl{2}{A}$ such that $\infty\cdot X=x_0$. Let $Y\in \spl{2}{A}$ satisfy
$Z:=X\cdot Y^{-1}\in B$. Let $a_1,\ldots,a_n$ be the elements associated to $X$ specifying the path $\mathbf{p}$  and let $b_1,\ldots, b_n$ be the corresponding elements associated to 
$Y$. Let $u=u(Z)$. Then
\[
b_1=a_1\cdot Z \mbox{ and }b_i=u^{(-1)^{i-1}2}a_i\mbox{ for all }i\geq 2.
\]
\end{prop}
\begin{proof}
By Corollary \ref{cor:key2} we have 
\[
E(b_i)\cdots E(b_1)=D(u)^{(-1)^{i-1}}E(a_i)\cdots E(a_1)Z\mbox{ in }\spl{2}{A}\mbox{ for all }i.
\]
and hence
\[
E(b_i)\cdots E(b_1)Y=D(u)^{(-1)^{i-1}}E(a_i)\cdots E(a_1)X \mbox{ in }\spl{2}{A}\mbox{ for all }i.
\]
\end{proof}

Thus we have the following algebraic description of paths in $Y(A)$, for any  ring $A$: 

Let $M_n:= \spl{2}{A}\times A^n$.  Given $p=(X,a_1,\ldots,a_n)\in M_n$ define 
\[
I(p):=T_0(p):=X \mbox{ and }  T_i(p):= E(a_i)\cdots E(a_1)X\in \spl{2}{A} \mbox{ for } i\geq 1
\]
and let $T(p):=T_n(p)$.  We define a relation $\sim$ on each of the sets $M_n$ as follows: 
\[
p=(X,a_1,\ldots,a_n)\sim (Y,b_1,\ldots,b_n)=q
\]
if and only if $Z:=X\cdot Y^{-1}=I(p)I(q)^{-1}\in B$ and 
\[
b_1=a_1\cdot Z\mbox{ and }b_i=u^{(-1)^{i-1}2}a_i\mbox{ for all }i\geq 2
\]
where $u=u(Z)$. It follows from Corollary \ref{cor:key2} that  $p\sim q$ if and only if  there exists $Z\in B$ such that 
\[
T_0(q)=Z^{-1}T_0(p)\mbox{ and } T_i(q)=D(u(Z))^{(-1)^{i-1}}T_i(p) \mbox{ for } i\geq 1.
\]

Proposition \ref{prop:changeX} thus shows that $\sim$ is an equivalence relation on $M_n$ for each $n$  and that $\mathcal{M}_n:=M_n/\sim$ is in bijective correspondence with 
the set of paths of length $n$ in $Y(A)$.
%We will verify below (Lemma \ref{lem:star} (1)) that $\sim$ is indeed an equivalence relation.
%Let $\mathcal{M}_n$ be the set of equivalence classes of $M_n$.
 If $p=(X,a_1,\ldots,a_n)\in M_n$ we will denote its image in $\mathcal{M}_n$ by $ [X,(a_1,\ldots,a_n)]$.
When $n=0$, $M_0=\spl{2}{A}$ and $\mathcal{M}_0=B\backslash\spl{2}{A}$. Note that $i([p]):= [I(p)], t([p]):=[T(p)]\in B\backslash\spl{2}{A}=\mathcal{M}_0$ are well-defined.

Thus the paths of $Y(A)$ are naturally in bijective correspondence with the elements of $\mathcal{M}:=\cup_{n=0}^\infty \mathcal{M}_n$: The path $\mathbf{p}=(x_0,\ldots,x_n)$ corresponds to the class $[X, (a_1,\ldots, a_n)]\in \mathcal{M}_n$ where $\infty\cdot X=x_0$ and for all $i\geq 1$,  $x_i=\infty E(a_i)\cdots E(a_1)X$, or equivalently,
\[
a_i=\infty\cdot E(a_i)=x_i\cdot X^{-1}E(a_1)^{-1}\cdots E(a_{i-1})^{-1}\mbox{ for }i\geq 1.
\]
Conversely, the class $[p]=[X,(a_1,\ldots,a_n)]$ corresponds to the path $(x_0,\ldots,x_n)$ where $x_i=\infty\cdot T_i(p)$ for all $i$.

 We now give an algebraic description of the operation of concatenating paths. We lift this operation to  the level of the sets $M_n$ as follows:

 Given $p=(X, a_1,\ldots, a_n)\in M_n, q=(Y,b_1,\ldots,b_m)\in M_m$, we define 
\[
Z_{p,q}:=I(q)T(p)^{-1}=Y(E(a_n)\cdots E(a_1)X)^{-1} = Y\cdot X^{-1}E(a_1)^{-1}\cdots E(a_n)^{-1}
\]
in $\spl{2}{A}$. (We will refer to this matrix below as the \emph{connection matrix for $p$ and $q$}.) Then $p\star q\in M_{n+m}$ is defined if $Z=Z_{p,q}\in B$ and is given by
\[
p\star q:= (X,a_1,\ldots,a_n, b'_1,\ldots,b'_m)
\]
where
\[
b'_1:=b_1\cdot Z\mbox{ and } b'_j:=u^{(-1)^{j-1}2}b_j\mbox{ for } 1\leq j\leq m
\]
where $u:=u_{p,q}=u(Z_{p,q})$ in this case. Observe that, by Corollary \ref{cor:key2}, the $b'_j$ are entirely determined by the requirement that 
\[
E(b'_j)\cdots E(b'_1)T(p)=D(u)^{(-1)^{j-1}}E(b_j)\cdots E(b_1)X'\mbox{ for } 1\leq j\leq m;
\]
i.e., 
\[
T_{j+n}(p\star q)=D(u_{p,q})^{(-1)^{j-1}}T_j(q) \mbox{ for } 1\leq j\leq m.
\]
In particular, taking $j=m$,  $T(p\star q)=D(u_{p,q})^{(-1)^{m-1}}T(q)$.

From the definition of  the operation $\star$, we have:
\begin{lem}\label{lem:dec}
Given any $p=(X,a_1,\ldots, n)\in M_n$ with $n>1$ and  given $i$ with $1<i<n$
\[
p=(X,a_1,\ldots,a_n)=(X,a_1,\ldots,a_{i-1})\star(T_{i-1}(p),a_i,\ldots,a_n).
\]
\end{lem}

By induction on $n$ we deduce immediately:
\begin{cor}\label{cor:dec}
For any $n\geq 1$ and any $p=(X,a_1,\ldots, n)\in M_n$ we have 
\[
p=\bigstar_{i=1}^n (T_{i-1}(p), a_i)
\]
\end{cor}

If $Y\in M_0=\spl{2}{A}$ and if $p= (X,a_1,\ldots,a_n)\in M_n$ then $Y\star p$  is defined if and only if $Z= XY^{-1}\in B$ and in this case $Y\star p =(Y,b_1,\ldots,b_n)$ satisfies
$T_i(q)=D(u(Z))^{(-1)^{n-1}}T_i(p)$ for $i\geq 1$.  Note that we always have $Y=I(Y\star p)$ whenever this exists. We immediately deduce:
\begin{lem}\label{lem:equiv}
 Let $p,q\in M_n$. Then $p\sim q$ if and only if  $I(q)\star p$ exists and equals $q$. 
\end{lem}

Furthermore, we observe, from the definitions, if $p\in M_n$ and $X\in \spl{2}{A}$, then $p\star X$ exists if and only if $XT(p)^{-1}\in B$ and, in this case, $p\star X=p$.

By Proposition \ref{prop:changeX} it follows that $\star$ induces a well-defined operation $[p]\star[q]$ on $\mathcal{M}$ which corresponds to concatenation of paths in $Y(A)$. In particular,
$\star$ is an assocative operation on $\mathcal{M}$. 

\begin{rem}
In fact, it is not hard to show that $\star$ is already associative on $M=\cup_n M_n$, and thus makes $M$ into a category,  but we will not need this below.
\end{rem}

To summarize: We have a well-defined category whose objects are $\mathcal{M}_0=B\backslash \spl{2}{A}$ and whose morphisms are $\mathcal{M}=\cup_{n=0}^\infty\mathcal{M}_n$.
For any $x,y\in\mathcal{M}_0$ and  $[p]\in \mathcal{M}$,we have $[p]\in \mathrm{Hom}(x,y)$ if and only if $i([p])=x$ and $t[p]=y$. When $t[p]=i[q]$ then composition is given by 
$[q]\circ [p]:=[p]\star [q]$. 

We also note that there is a natural right action of $\spl{2}{A}$ on this category corresponding to the action of $\spl{2}{A}$ on paths of $\Gamma(A)$ by right multiplication. Te action is described as follows:   Given $Y\in \spl{2}{A}$ and $p=(X,a_1,\ldots,a_n)\in M_n$, we have
\[
[p]\cdot Y=[X,a_1,\ldots,a_n]\cdot Y:= [XY,a_1,\ldots,a_n].
\]

\subsection{Homotopy relation (1)}
We use the correspondence between paths in $\Gamma(A)$ and elements of $\mathcal{M}$ to express the homotopy relation (1) of the edge-path groupoid of $Y(A)$ in terms of the 
category $\mathcal{M}$. 

First note that if $\{ x,y\} $ is an edge in $\Gamma(A)$ and if $x=\infty\cdot X$ for some $X\in \spl{2}{A}$, then $y=a\cdot X=\infty\cdot E(a)X$ for some $a\in A$. Now
$x\cdot X^{-1}E(a)^{-1}=\infty\cdot E(a)^{-1}=0=\infty\cdot E(0)\implies x=\infty\cdot E(0)E(a)X$. So the path $(x,y,x)$ corresponds to $[X, (a,0)]=[1,(a,0)]\cdot X\in \mathcal{M}_2$.
So the relation derived from homotopy relation (1) has the form $[X,(a,0)]\sim_{(1)}[X]$  and, more generally, for $1\leq i\leq n-1$
\begin{eqnarray*}
&&[X,(a_1,\ldots, a_{i-1},a_i,0,a_{i+2},\ldots, a_n)]\\
&=&[X,(a_1,\ldots,a_{i-1})]\star [T_{i-1}(p), (a_i,0)]\star[T_{i+1}(p),(a_{i+2},\cdots a_n)]\\
&\sim_{(1)}& [X,(a_1,\ldots,a_{i-1})]\star [T_{i-1}(p)]\star[T_{i+1}(p),(a_{i+2},\cdots a_n)]\\
&=& [X,(a_1,\ldots,a_{i-1})]\star[T_{i+1}(p),(a_{i+2},\cdots a_n)]\\
\end{eqnarray*}
Now the connection matrix in this case is 
\[
Z=T_{i+1}(p)T_{i-1}(p)^{-1}=E(0)E(a_i) =
\left[
\begin{array}{cc}
-1&0\\
-a_i&-1
\end{array}
\right]\in B
\]
with unit $u(Z)=-1$. Thus $a_{i+2}Z= a_i+a_{i+2}$ and we conclude that
\begin{eqnarray*}
[X,(a_1,\ldots,a_{i-1},\underbrace{ a_i,0,a_{i+2}},\ldots, a_n)]\quad\quad\quad\quad\quad\quad\quad\quad\quad\quad\quad\quad\quad\\
\sim_{(1)}\left\{
\begin{array}{ll}
\ [X,(a_1,\ldots,a_{i-1})],& \mbox{ if } i+1=n\\
\ [X,(a_1,\ldots,a_{i-1},\underbrace{a_i+a_{i+2}},a_{i+3},\ldots,a_n)],& \mbox{ if } i+1<n\\
\end{array}
\right.
\end{eqnarray*}

We spell out the complete  equivalence relation $\sim_{(1)}$ on $\mathcal{M}$ in greater detail:

 Let us say that replacing $[X,(a_1,\ldots, a_{i-1},a_i,0,a_{i+2},\ldots, a_n)]$ with
$[X,(a_1,\ldots,a_{i-1},a_i+a_{i+2},a_{i+3},\ldots,a_n)]$ (or replacing $[X,(a_1,\ldots,a_{n-1},0)]$ with $[X,(a_1,\ldots,a_{n-2})]$) is a \emph{type (1) contraction}. 

Conversely, replacing $[X, (a_1,\ldots,a_{i-1},  a_i,\ldots, a_n)]$ with $[X, (a_1,\ldots, a_{i-1},a,0,a_{i}-a,\ldots,a_n)]$  (or with $[X,(a_1,\ldots,a_n,a,0)]$ ) for any $a\in A$ will be referred to as a \emph{type (1) expansion}.

Then $[p]\sim_{(1)}[q]$ if and only if $[q]$ can be obtained from $[p]$ by a finite sequence of replacements, each of which is either a type (1) contraction or a type (1) expansion.

We note that  modulo the equivalence relation $\sim_{(1)}$ the category $\mathcal{M}$ becomes a groupoid (indeed, the fundamental groupoid of the graph $\Gamma(A)$). For example, we have 
\[
[X,(a)]\star [E(a)X,(0)]=[X, (a,0)]\sim_{(1)} [X]
\]
and thus $[E(a)X,(0)]$ is a right inverse of $[X,(a)]$ in the quotient category $\mathcal{M}/\sim_{(1)}$. 
\subsection{Homotopy relation (2)}
Now, in a similar manner,  we transfer the homotopy relation (2) to the category $\mathcal{M}$:

Let $\{ x,y,z\}$ be a $3$-clique in $\Gamma(A)$, and hence a $2$-simplex in $Y(A)$. Choose $X\in \spl{2}{A}$ with $x=\infty\cdot  X$. Since $y,z$ are neighbours of $x$, there exist 
$a,b\in A$ with $y=a\cdot X$ and $z= b\cdot X$. Since $y$ and $z$ (and hence $a$ and $b$) are neighbours there exists $u\in A^\times$ with $b=u+a$. Thus 
$y=\infty \cdot E(a)X$, $z=\infty\cdot E(b)X=\infty\cdot E(a+u)X$.   Furthermore, 
\[
z\cdot X^{-1}E(a)^{-1}=\infty\cdot E(u+a)E(a)^{-1}=-u^{-1}=\infty\cdot E(-u^{-1})
\]
 and hence 
$z=\infty\cdot E(u^{-1})E(a)X$

Thus the path $\mathbf{p}_1:=(x,y,z)$ corresponds to the class $[X,(a,-u^{-1})]\in  \mathcal{M}$ and the path $\mathbf{p}_2:=(x,z)$ corresponds to the class $[X,(a+u)]$,
Since $\mathbf{p}_1$ and $\mathbf{p}_2$ are identified by the second homotopy relation, we must have $[X,(a,-u^{-1})]\sim_{(2)}[X,(a+u)]$ for any $X\in\spl{2}{A}$, 
$a\in A$ and $u\in A^\times$, or,  equivalently, $[X,(a,u)]\sim_{(2)}[X,(a-u^{-1})]$ for all $X,a,u$. More generally, we have
\begin{eqnarray*}
&&[X,(a_1,\ldots, a_{i-1},a_i,0,a_{i+2},\ldots, a_n)]\\
&=&[X,(a_1,\ldots,a_{i-1})]\star [T_{i-1}(p), (a_i,u)]\star[T_{i+1}(p),(a_{i+2},\cdots a_n)]\\
&\sim_{(1)}& [X,(a_1,\ldots,a_{i-1})]\star [T_{i-1}(p),(a_i-u^{-1})]\star[T_{i+1}(p),(a_{i+2},\cdots a_n)]\\
&=&  [X,(a_1,\ldots,a_{i-1}, a_i-u^{-1})]\star[T_{i+1}(p),(a_{i+2},\cdots a_n)].\\
\end{eqnarray*}
The connection matrix $Z$ here is 
\[
T_{i+1}(p)\left(E(a_i-u^{-1})T_{i-1}(p)\right)^{-1}= E(u)E(a_i)E(a_i-u^{-1})^{-1}=
\left[
\begin{array}{cc}
u&0\\
-1&u^{-1}
\end{array}
\right]\in B.
\]
So $u(Z)=u$ and $a_{i+2}\cdot Z= u^2a_{i+2}-u$ and we get:
\begin{eqnarray*}
[X,(a_1,\ldots, a_{i-1},a_i,u,a_{i+2},\ldots, a_n)]\quad\quad\quad\quad\quad\quad\quad\quad\quad\quad\\
\sim_{(2)}
\left\{
\begin{array}{ll}
\ [X,(a_1,\ldots,a_{i-1},a_i-u^{-1})],& \mbox{ if } i+1=n\\
\ [X,(a_1,\ldots,a_{i-1},a_i-u^{-1},u^2a_{i+2}-u,u^{-2}a_{i+3},u^2a_{i+4},\ldots],& \mbox{ if } i+1<n\\
\end{array}
\right.\\
\end{eqnarray*}

Let us define a a \emph{type (2) contraction} to be a replacement of $[X,(a_1,\ldots, a_{i-1},a_i,u,a_{i+2},\ldots, a_n)]$ with  $ [X,(a_1,\ldots,a_{i-1},a_i-u^{-1},u^2a_{i+2}+u,u^{-2}a_{i+3},u^2a_{i+4},\ldots]$ 
(or of $ [X,(a_1,\ldots, a_{n-2},a_{n-1},u)]$ with $[X,(a_1,\ldots,a_{n-2},a_{n-1}-u^{-1})]$) .

Conversely,  replacing $[X,(a_1,\ldots,a_i,a_{i+2},a_{i+2},\cdots)]$ with $[X,(a_1,\ldots,a_{i-1},a_i+u^{-1},u,u^{-2}a_{i+1}-u^{-1},u^{-2}a_{i+2},\ldots)]$ (or replacing
$[X,(a_1,\ldots,a_n)]$ with $[X,(a_1,\ldots,a_{n-1},a_n+u^{-1},u]$) will be called a \emph{type (2) expansion}. We have $[p]\sim_{(2)}[q]$ if $[q]$ can be obtained from $[p]$ by a finite 
sequence of replacements, each of which is a type (2) contraction or a type (2) expansion. 

Thus the edge-path groupoid $\mathcal{E}(Y(A))$ is naturally isomorphic to  the quotient of the category $\mathcal{M}$ by the equivalence relations $\sim_{(1)}$ and $\sim_{(2)}$. 
%%%%%%%%%%%%%%%%%%%%%%%%%%%%%%%%%%%%%%%%%%%%%%%%%%%%%%%%%%%%%%%%%%%
\section{The fundamental group of $Y(A)$}
%%%%%%%%%%%%%%%%%%%%%%%%%%%%%%%%%%%%%%%%%%%%%%%%%%%%%%%%%%%%%%%%%%%%%%
\subsection{A presentation for $\pi_1(Y(A),\infty)$}
Let $A$ be a ring. Let $\mathcal{E}:=\mathcal{E(Y(A))}$ be the edge-path groupoid of $Y(A)$, identified as a quotient of the category $\mathcal{M}$. We are now in a position to give a presentation of the fundamental group of $Y(A)$ at $\infty$; $\pi_1(Y(A),\infty):=\mathrm{Aut}_{\mathcal{E}}(\infty)$.

Observe that $[p]\in M_n$ represents an automorphism of $\infty$ if and only if $\infty\cdot I(p)=\infty\cdot T(p)=\infty$, and hence if and only if $I(p),T(p)\in B$. Let 
$M_n(B)=\{ p\in m_n\ |\ I(p),T(p)\in B\}$ and let $M(B):=\cup_{n=0}^\infty M_n(B)$. Note that if $p,q\in M(B)$, then $p\star q$ is always defined. 
Let $\mathcal{M}(B)$ be the image of $M(B)$ in $\mathcal{M}$. Thus  $\mathrm{Aut}_{\mathcal{E}}(\infty)$ is the image of $\mathcal{M}(B)$ in $\mathcal{E}$. Using the description of the 
edge-path groupoid $\mathcal{E}$ in the last section, as well as the relationship between the operation $\star$ on $M$ and the equivalence relation $\sim$, we deduce the following 
presentation of $\pi_1(Y(A),\infty)$.

\begin{thm}\label{thm:pi1} $\pi_1(Y(A),\infty)$ is generated by the symbols $p=\an{X,(a_1,\ldots,a_n)}$ $n\geq 0$, where $X\in B$, and (if $n\geq 1$) $a_1,\ldots,a_n\in A$ satisfy $E(a_n)\cdots E(a_1)\in B$. Given such a $p=\an{X,(a_1,\ldots,a_n)}$ we let $I(p):=T_0(p):=X$ , $T_i(p):=E(a_i)\cdots E(a_1)X$ for $i\geq 1$ and $T(p):=T_n(p)$. 
These symbols are subject to the following families of relations:
\begin{enumerate}
\item Given generators $p=\an{X,(a_1,\ldots,a_n)}$, $q={Y,(b_1,\ldots,b_m)}$ we have 
\[
\an{X,(a_1,\ldots,a_n)}\cdot \an{Y,(b_1,\ldots,b_m)}=\an{X,(a_1,\cdots,a_n,b'_1,\ldots,b'_n)}
\]
where $b'_1=b_1\cdot Z$, $b'_j=u^{(-1)^{j-1}2}b_j$ for $j\geq 2$. Here $Z=Z_{p,q}=I(q)T(p)^{-1}=YX^{-1}E(a_1)^{-1}\cdots E(a_n)^{-1}\in B$ and $u=u(Z)$.
\item $\an{X}=1$ for all $X\in B$. 
\item (Type (1) contractions)
 $\an{X, (a_1,\ldots, a_{n-1},0)}=\an{X,(a_1,\ldots,a_{n-2})}$ and 
\[
\an{X,(a_1,\ldots,a_i,0,a_{i+2},\ldots)}=\an{X,(a_1,\ldots,a_i+a_{i+2},a_{i+3},\ldots)}.
\]
\item (Type (2) contractions)
$\an{X,(a_1,\ldots,a_{n-1},u)}=\an{X,(a_1,\ldots,a_{n-1}-u^{-1})}$ and 
\[
\an{X,(a_1,\ldots,a_i, u,a_{i+2},a_{i+3},\ldots)}=\an{X,(a_1,\ldots,a_{i}-u^{-1},u^2a_{i+2}-u,u^{-2}a_{i+3},\ldots)}.
\]
\end{enumerate}
\end{thm}

\subsection{Proof of the main theorem}
Recall that $\mathrm{st}:B\to C(A)$ denotes the standard embedding with image the subgroup $\mathbb{B}$ satisfying $\psi\circ\mathrm{st}=\mathrm{Id}_B$.
Given a symbol $p=\an{X,(a_1,\ldots,a_n)}$ we define $\Lambda(p)\in C(A)$ by the formula
\[
\Lambda(p):=\stand{T(p)}^{-1}\epsilon(a_n)\cdots\epsilon(a_1)\stand{I(p)}.
\]

Observe that $\psi(\Lambda(p))= T(p)^{-1}E(a_n)\cdots E(a_1)I(p)=T(p)^{-1}T(p)=1$, so that $\Lambda(p)\in U(A)$ for all $p$. 
To begin with, we will show that $\Lambda$ determines a group anti-homomorphism $\pi_1(Y(A),\infty)\to U(A)$. We must verify that $\Lambda$ respects the relations (1)-(4). 

\begin{lem}\label{lem:rel1}
Given $p=\an{X,(a_1,\ldots,a_n)}$, $q=\an{Y,(b_1,\ldots,b_m)}$ we have \[
\Lambda(p\cdot q)=\Lambda(q)\cdot \Lambda(p)
\]
 where $p\cdot q$ is given by relation (1) of Theorem \ref{thm:pi1}.
\end{lem}
\begin{proof}
By Proposition \ref{prop:key}, we have, setting $\beta_{p,q}:=\stand{Z_{p,q}}\in\mathbb{B}$,
\begin{eqnarray*}
\epsilon(b'_m)\cdots \epsilon(b'_1)&=&h(u)^{(-1)^{m-1}}\epsilon(b_m)\cdots \epsilon(b_1)\beta_{p,q}\\
&=& h(u)^{(-1)^{m-1}}\epsilon(b_m)\cdots \epsilon(b_1)\stand{I(q)}\stand{T(p)}^{-1}\\
\end{eqnarray*}
and hence
\[
\epsilon(b'_m)\cdots \epsilon(b'_1)\stand{T(p)}=h(u)^{(-1)^{m-1}}\epsilon(b_m)\cdots \epsilon(b_1)\stand{I(q)}.
\]
Applying $\psi$ this also gives
\[
T(p\cdot q)=E(b'_m)\cdots E(b'_1)T(p)=D(u)^{(-1)^{m-1}}E(b_m)\cdots E(b_1)I(q)=D(u)^{(-1)^{m-1}}T(q).
\]

Hence $T(p\cdot q)^{-1}=T(q)^{-1}D(u)^{(-1)^m}$ and $\stand{T(p\cdot q)}^{-1}=\stand{T(q)}^{-1}h(u)^{(-1)^m}$.

Thus 
\begin{eqnarray*}
\Lambda(q)\Lambda(p)&=&\stand{T(q)}^{-1}\epsilon(b_m)\cdots \epsilon(b_1)\stand{I(q)}\stand{T(p)}^{-1}\epsilon(a_n)\cdots \epsilon(a_1)\stand{I(p)}\\
&=& \stand{T(q)}^{-1}\epsilon(b_m)\cdots \epsilon(b_1)\beta_{p,q}\epsilon(a_n)\cdots \epsilon(a_1)\stand{I(p)}\\
&=&  \stand{T(q)}^{-1}h(u)^{(-1)^m}\epsilon(b'_m)\cdots \epsilon(b'_1)\epsilon(a_n)\cdots \epsilon(a_1)\stand{I(p)}\\
&=&  \stand{T(p\cdot q)}^{-1}\epsilon(b'_m)\cdots \epsilon(b'_1)\epsilon(a_n)\cdots \epsilon(a_1)\stand{I(p)}\\
&=& \Lambda(p\cdot q)
\end{eqnarray*}
as required.
\end{proof}

If $p=\an{X}$, we have $I(p)=X=T(p)$ and hence $\Lambda(\an{X})=1$:
\begin{lem}\label{lem:rel2} For all $X\in B$, $\lambda(\an{X})=1$ in $U(A)$.
\end{lem}

\begin{lem}\label{lem:rel3} Let $X\in B$.
\begin{enumerate}
\item Suppose that $a_1,\ldots,a_i,a_{i+2},\ldots,a_n\in A$ satisfy $E(a_n)\cdots E(a_{i+2})E(0)E(a_i)\cdots E(a_1)\in B$. Then $\Lambda(p)=\Lambda(q)$ in $U(A)$ where 
$p=\an{X,(a_1,\ldots,a_i,0,a_{i+2},\ldots,a_n)}$ and $q=\an{X,(a_1,\ldots,a_{i-1},a_i+a_{i+2},\ldots,a_n)}$. 
\item Suppose that $a_1,\ldots,a_{n-1}\in A$ satisfy $E(0)E(a_{n-1})\cdots E(a_1)\in B$. Then $\Lambda(p)=\Lambda(q)$ in $U(A)$ where 
$p=\an{X,(a_1,\ldots,a_{n-1},0)}$ and $q=\an{X,(a_1,\ldots,a_{n-2})}$. 
\end{enumerate}
\end{lem}
\begin{proof} Recall that $\epsilon(a)\epsilon(0)\epsilon(b)=h(-1)\epsilon(a+b)$ in $C(A)$ and that $E(a)E(0)E(b)=D(-1)E(a+b)$ in $E_2(A)$. 
\begin{enumerate}
\item 
Note that $I(p)=I(q)$ and 
\begin{eqnarray*}
T(p)&=& E(a_n)\cdots \left(E(a_{i+2})E(0)E(a_i)\right)\cdots E(a_1)X\\
&=& D(-1)E(a_n)\cdots E(a_i+a_{i+2})\cdots E(a_1)X\\
&=& D(-1)T(q).
\end{eqnarray*}
It follows that $\stand{T(p)}=h(-1)\stand{T(q)}$ in $ C(A)$. Thus (recalling that $h(-1)$ is central in $C(A)$)
\begin{eqnarray*}
\Lambda(p)&=&\stand{T(p)}^{-1}\epsilon(a_n)\cdots \left(\epsilon(a_{i+2})\epsilon(0)\epsilon(a_i)\right)\cdots \epsilon(a_1)\stand{I(p)}\\
&=& \stand{T(p)}^{-1}h(-1)\epsilon(a_n)\cdots \epsilon(a_i+a_{i+2})\cdots \epsilon(a_1)\stand{I(p)}\\
&=& \stand{T(q)}^{-1}h(-1)^2\epsilon(a_n)\cdots \epsilon(a_i+a_{i+2})\cdots \epsilon(a_1)\stand{I(q)}\\
&=& \stand{T(q)}^{-1}\epsilon(a_n)\cdots \epsilon(a_i+a_{i+2})\cdots \epsilon(a_1)\stand{I(q)}\\
&=&\Lambda(q).
\end{eqnarray*}
\item We have, by definition of $y(a)$ and $\beta(u,a)$,  $\epsilon(-1)\epsilon(a_{n-1})=h(-1)y(a_{n-1})=\beta(-1,-a_{n-1})$ in $\mathbb{B}\subset C(A)$  and $E(0)E(A_{n-1})=D(-1)E_{21}(a_{n-1})$ 
in $B$. It follows that  $\stand{E(0)E(a_{n-1}}=\epsilon(0)\epsilon(a_{n-1})$. 

Now $T(p)=E(0)E(a_{n-1}) E(a_{n-2})\cdots E(a_1)X=E(0)E(a_{n-1})T(q)$. Hence $\stand{T(p)}= \epsilon(0)\epsilon(a_{n-1})\stand{T(q)}$ in $\mathbb{B}$. It follows that 
\begin{eqnarray*}
\Lambda(p)&=& \stand{T(p)}^{-1}\epsilon(0)\epsilon(a_{n-1})\epsilon(a_{n-2})\cdots \epsilon(a_1)\stand{I(p)}\\
&=& \stand{T(q)}^{-1}\epsilon(a_{n-2})\cdots \epsilon(a_1)\stand{I(q)}\\
&=& \Lambda(q).
\end{eqnarray*}
\end{enumerate}
\end{proof}

\begin{lem}\label{lem:rel4} Let $X\in B$.  Let $u\in A^\times$.
\begin{enumerate}
\item Suppose that $a_1,\ldots,a_i,a_{i+2},\ldots,a_n\in A$ satisfy $E(a_n)\cdots E(a_{i+2})E(u)E(a_i)\cdots E(a_1)\in B$. 
Then $\Lambda(p)=\Lambda(q)$ in $U(A)$ where 
$p=\an{X,(a_1,\ldots,a_i,u,a_{i+2},\ldots,a_n)}$ and $q=\an{X,(a_1,\ldots,a_{i-1},a_i-u^{-1},u^2a_{i+2}-u,u^{-2}u_{i+3},\cdots)}$. 
\item Suppose that $a_1,\ldots,a_{n-1}\in A$ satisfy $E(u)E(a_{n-1})\cdots E(a_1)\in B$. Then $\Lambda(p)=\Lambda(q)$ in $U(A)$ where 
$p=\an{X,(a_1,\ldots,a_{n-1},u)}$ and $q=\an{X,(a_1,\ldots,a_{n-2},a_{n-1}-u^{-1})}$. 
\end{enumerate}
\end{lem}
\begin{proof} 
\begin{enumerate}
\item 
In $\mathbb{B}$ we have  
\begin{eqnarray*}
\beta(-u,1)=h(-u)y(-u)=\epsilon(u)\epsilon(u^{-1})\epsilon(u)\epsilon(0)^3\epsilon(-u)\\
=h(-1)\epsilon(u)\epsilon(u^{-1})\epsilon(u)\epsilon(0)\epsilon(-u)
=h(-1)\epsilon(u)\epsilon(u^{-1})\epsilon(0)
\end{eqnarray*}
using relation (2) in the definition of $C(A)$. Thus consider $\beta:=h(-1)\beta(-u,1)=\epsilon(u)\epsilon(u^{-1})\epsilon(0)\in \mathbb{B}$. Then
\[
\beta\epsilon(a_i-u^{-1})=\epsilon(u)\epsilon(u^{-1})\epsilon(0)\epsilon(a_i-u^{-1})=\epsilon(u)\epsilon(a_i) \mbox{ in  }C(A).
\]
Furthermore, if we let $Z:=\psi(\beta)\in B$, then it follows that $ZE(a_i-u^{-1})=E(u)E(a_i)$ in $B$. 

By Proposition \ref{prop:key} we have
\[
\epsilon(a'_n)\cdots \epsilon(a'_{i+2})=h(u)^{(-1)^{n-i}}\epsilon(a_n)\epsilon(a_{i+2})\beta
\]
where $a'_{i+2}:=u^2a_{i+2}-u=a_{i+2}\beta$, $a'_{i+3}=u^{-2}a_{i+3},\ldots$ etc. We thus also have
\[
E(a'_n)\cdots E(a'_{i+2})=D(u)^{(-1)^{n-i}}E(a_n)\cdots E(a_{i+2})Z.
\]

Hence 
\begin{eqnarray*}
T(q)&=& E(a'_n)\cdots E(a'_{i+2})E(a_i-u^{-1})T_{i-1}(p)\\
&=& D(u)^{(-1)^{n-i}}E(a_n)\cdots E(a_{i+2})ZE(a_i-u^{-1}T_{i-1}(p)\\
&=&  D(u)^{(-1)^{n-i}}E(a_n)\cdots E(a_{i+2})E(u)E(a_i)T_{i-1}(p)\\
&=& D(u)^{(-1)^{n-i}}T(p). 
\end{eqnarray*}

So
\begin{eqnarray*}
\Lambda(q)&=& \stand{T(q)}^{-1}\epsilon(a'_n)\cdots \epsilon(a'_{i+2})\epsilon(a_i-u^{-1})\cdots \epsilon(a_1)\stand{X}\\
&=& \stand{T(p)}^{-1}h(u)^{(-1)^{n-i+1}}h(u)^{(-1)^{n-i}}\epsilon(a_n)\epsilon(a_{i+2})\beta\epsilon(a_i-u^{-1})\cdots \epsilon(a_1)\stand{X}\\
&=& \stand{T(p)}^{-1}\epsilon(a_n)\epsilon(u)\epsilon(a_i)\cdots \epsilon(a_1)\stand{X}\\
&=& \Lambda(p).
\end{eqnarray*}
\item  Let $\beta$, $Z$ be as in (1), taking $i=n-1$. Then $\epsilon(u)\epsilon(a_{n-1})=\beta\epsilon(a_{n-1}-u^{-1})$ and $E(u)E(a_{n-1})=ZE(a_{n-1}-u^{-1})$. 
Thus  
\[
T(p)=E(u)E(a_{n-1})T_{n-2}(p)=Z E(a_{n-1}-u^{-1})T_{n-2}(q)=ZT(q)
\]
 and so $\stand{T(p)}=\beta\stand{T(q)}$.  Thus
\begin{eqnarray*}
\Lambda(p)&=& \stand{T(p)}^{-1}\epsilon(u)\epsilon(a_{n-1})\cdots\epsilon(a_1)\stand{X}\\
&=&\stand{T(q)}^{-1}\beta^{-1}\cdot\beta\epsilon(a_{n-1}-u^{-1})\cdots\epsilon(a_1)\stand{X}\\
&=& \Lambda(q).
\end{eqnarray*}
\end{enumerate}
\end{proof}

 Lemmas \ref{lem:rel1}--\ref{lem:rel4} imply that $\Lambda$ defines a group anti-homomorphism $\pi_1(Y(A),\infty)\to U(A)$   (which, composed with the map $x\mapsto x^{-1}$ then yields a homomorphism, of course). We now proceed to construct an inverse, $\Theta$, to
$\Lambda$: 

Let $\tilde{C}(A)$ denote the free group with generators $\tilde{\epsilon}(a), a\in A$. Let $\tilde{\psi}:\tilde{C}(A)\to \spl{2}{A}$ be the homomorphism sending 
$\tilde{\epsilon}(a)$ to $E(a)$ for all $a\in A$. Let $\tilde{U}(A):=\ker{\tilde{\psi}}$. We begin by defining a map of \emph{sets} $\tilde{\Theta}:\tilde{U}(A)\to \pi_1(Y(A),\infty)$.

An element of $\tilde{C}(A)$ is an expression $\alpha=\tilde{\epsilon}(b_1)^{c_1}\cdots\tilde{\epsilon}(b_n)^{c_n}$ where $c_i\in \{ \pm 1\}$. Such an element $\alpha$ lies in 
$\tilde{U}(A)$ if and only if $\prod_{i=1}^nE(b_i)^{c_i}=1$ in $\spl{2}{A}$.

Now, for $b\in A$, we let $s_1(b)=b$ and we define $s_{-1}(b)$ to be the string $0,-b,0$. Given $\alpha=\prod_{i=1}^n\tilde{\epsilon}(b_i)^{c_i}\in \tilde{U}(A)$ we define
\[
\tilde{\Theta}(\alpha):= \an{1,(s_{c_n}(b_n),\ldots,s_{c_1}(b_1))}\in \pi_1(Y(A),\infty).
\]

We note that $\tilde{\Theta}$ is well-defined:
\begin{lem}\label{lem:thetadef} Let $\alpha=\prod_{i=1}^n\tilde{\epsilon}(b_i)^{c_i}\in \tilde{U}(A)$ and suppose that for some $i<n$ we have $b_i=b_{i+1}$ and 
$c_i=-c_{i+1}$. Then $\alpha=\alpha':= \tilde{\epsilon}(b_1)^{c_1}\cdots\tilde{\epsilon}(b_{i-1})^{c_{i-1}}\tilde{\epsilon}(b_{i+2})^{c_{i+2}}\cdots \tilde{\epsilon}(b_n)^{c_n}$
in $\tilde{U}(A)$ and $\tilde{\Theta}(\alpha)=\tilde{\Theta}(\alpha')$.
\end{lem}

\begin{proof} We consider the case $c_i=1,c_{i+1}=-1$. The other case is entirely similar.
\begin{eqnarray*}
\tilde{\Theta}(\alpha)&=&\an{1,(\ldots, s_{c_{i+2}}(b_{i+2}), 0,\underbrace{-b_{i},0,b_i},s_{c_{i-1}}(b_{i-1}),\ldots)}\\
&=& \an{1,(\ldots, s_{c_{i+2}}(b_{i+2}),\underbrace{ 0,0,s_{c_{i-1}}(b_{i-1})},\ldots)}\\
&=& \an{1,(\ldots, s_{c_{i+2}}(b_{i+2}),s_{c_{i-1}}(b_{i-1}),\ldots)}=\tilde{\Theta}(\alpha')\\
\end{eqnarray*}
using type (1) contractions in each of the last two lines.
\end{proof}

$\tilde{\Theta}$ defines an anti-homomorphism of groups:
\begin{lem}\label{lem:thetaprod}
 For all $\alpha,\beta\in \tilde{U}(A)$ we have $\tilde{\Theta}(\alpha\beta)=\tilde{\Theta}(\beta)\tilde{\Theta}(\alpha)$.
\end{lem}
\begin{proof}  Let $\alpha=\prod_{i=1}^n\tilde{\epsilon}(b_i)^{c_i}$ and let $\beta=\prod_{i=1}^n\tilde{\epsilon}(d_i)^{e_i}$. Then 
\begin{eqnarray*}
\tilde{\Theta}(\beta)\tilde{\Theta}(\alpha)&=&\an{1,(s_{e_n}(d_n),\ldots,s_{e_1}(d_1))}\an{1,(s_{c_n}(b_n),\ldots,s_{c_1}(b_1))}\\
&=&\an{1,(s_{e_n}(d_n),\ldots,s_{e_1}(d_1),s_{c_n}(b_n),\ldots,s_{c_1}(b_1))}=\tilde{\Theta}(\alpha\beta).\\
\end{eqnarray*}
\end{proof}

Now $U(A)$ is isomorphic to $\tilde{U}(A)$ modulo the normal subgroup generated by the following three families of elements: 

 For $u\in A^\times$, let $\tilde{h}(u):=\tilde{\epsilon}(-u)\tilde{\epsilon}(-u^{-1})\tilde{\epsilon}(-u)$.
\begin{enumerate} 
\item For $u,v\in A^\times$
\[
\alpha(u,v):=\tilde{h}(u)\tilde{h}(u)\tilde{h}(uv)^{-1}
\]
\item For $a,b\in A$
\[
\gamma(a,b):= \tilde{\epsilon}(a)\tilde{\epsilon}(0)\tilde{\epsilon}(b)\tilde{\epsilon}(a+b)^{-1}\tilde{h}(-1)^{-1}.
\]
\item For $u\in A^\times$, $a\in A$
\[
\delta(u,a):=\tilde{h}(u)\tilde{\epsilon}(a)\tilde{h}(u)\tilde{\epsilon}(u^2a)^{-1}.
\]
\end{enumerate}

\begin{prop}\label{prop:theta} 
$\tilde{\Theta}$ induces a well-defined anti-homomorphism 
\[
\Theta:U(A)\to \pi_1(Y(A),\infty).
\]
\end{prop}

\begin{proof} We must show that $\tilde{\Theta}$ vanishes on each of the three families  -- $\alpha$, $\gamma$ ,$\delta$ -- of elements of $\tilde{U}(A)$.
\begin{enumerate}
\item Let $u,v\in A^\times$. Then $\tilde{\Theta}(\alpha(u,v))=$
\begin{eqnarray*}
 \an{1,(0,uv,0,0,(uv)^{-1},0,0,uv,0,-v,-v^{-1},-v,-u,-u^{-1},-u)}.
\end{eqnarray*}
Since a type (1) contraction replaces the string $0,0,a$ with $a$, this is equal to
\[
 \an{1,(0,uv,(uv)^{-1},uv,0,-v,-v^{-1},-v,-u,-u^{-1},-u)}.
\]
Since a type (2) contraction replaces the string $u,u^{-1},u$ with $0,0$ (and multiplies terms further to the right by powers of $u^2$), making three such contractions,
 starting on the right for convenience, gives the element
\[
\an{1,(0,0,0,0,0,0,0,0)}
\]
which is trivial by further type (1) contractions. 
\item Let $a,b\in A$. Then $\tilde{\Theta}(\gamma(a,b))=$
\[
\an{1,(0,-1,0,0,-1,0,0,-1,0,0,-(a+b),0,b,0,a)}.
\]
Replacing $0,0,x$ with $x$ in three places, this is equal to
\[
\an{1,(0,-1,-1,-1,-(a+b),0,b,0,a)}.
\]
A type (2) contraction replaces $-1,-1,-1$ with $0,0$, and leaves terms to the right unaltered  since $(-1)^2=1$). So our element becomes
$\an{1,(0,0,0,-(a+b),0,b,0,a)}$, which is again trivial by a sequence of type (1) contractions. 
\item Let $u\in A^\times$, $a\in A$.  Then $\tilde{\Theta}(\delta(u,a))=$
\[
\an{1,(0,-u^2a,0,-u,-u^{-1},-u,a,-u,-u^{-1},-u)}.
\]
Applying a type (2) contraction to the last three terms, this is equal to \[
\an{1,(0,-u^2a,0,-u,-u^{-1},-u,a,0,0)}. 
\]
Applying a further type (2) contraction, this becomes $\an{1,(0,-u^2a,0,0,0,u^2a,0,0)}$, which 
is trivial by a sequence of type (1) contractions.
\end{enumerate}
\end{proof}

\begin{thm}\label{thm:main} Let $A$ be a ring. The anti-homomophisms $\Lambda:U(A)\to\pi_1(Y(A),\infty)$ and $\Theta:\pi_1(Y(A),\infty)\to U(A)$ are inverse to each other, and hence
\[
\pi_1(Y(A),\infty)\cong U(A)\cong \frac{K_2(2,A)}{C(2,A)}.
\]
\end{thm}
\begin{proof}
 Let $p=\an{X,(a_1,\ldots,a_n)}\in \pi_1(Y(A),\infty)$. Then $p=\an{1}\cdot p =\an{1,(b_1,\ldots,b_n)}$ for some $b_i\in A$. Thus 
$\Lambda(p)=\epsilon(b_n)\cdots \epsilon(b_1)$ and hence $\Theta(\Lambda(p))=\an{1,(b_1,\ldots,b_n)}=p$.

Let $\alpha\in U(A)$. Since $\epsilon(a)^{-1}=\epsilon(0)\epsilon(-a)\epsilon(0)$ in $U(A)$, we can assume without loss that $\alpha=\prod_{i=1}^n\epsilon(b_i)$ for some 
$b_1,\ldots,b_n\in A$. Then $\Theta(\alpha)=\an{1,(b_n,\ldots, b_1)}\in \pi_1(Y(A),\infty)$ and hence $\Lambda(\Theta(\alpha))=\prod_{i=1}^n\epsilon(b_i)=\alpha$, as required.
Thus $\Theta$ defines an anti-isomorphism, and composing with the inversion involution thus gives an isomorphism.

The second stated isomorphism is Theorem \ref{thm:gamma} in Appendix \ref{sec:k2} below.
\end{proof}

Combining this with Proposition \ref{prop:euclid}  (and recalling that  $\mathrm{GL}_2(A)$ acts transitively on the set of path components of $|Y(A)|$) we deduce:
\begin{cor}\label{cor:main} Let $A$ be a ring. Then $A$ is  universal for $\mathrm{GE}_2$ if and only if any path component of the space $|Y(A)|$ is simply-connected. 
\end{cor}

\subsection{Some examples}
\begin{exa} \label{exa:pi1f}
For any field $F$, the space $|Y(F)|$ is contractible (see Example \ref{exa:field} above) and hence $\pi_1(Y(F),\infty)=1$. The well-known fact that $K_2(2,F)$ is generated by symbols and the resulting presentation of $\spl{2}{F}$  follow from this.
\end{exa}

\begin{exa} \label{exa:pi1z}
The space $|Y(\Z)|$ is contractible (see Section \ref{sec:gammaz} above) and hence $\pi_1(Y(\Z),\infty)=1$. 
Thus $K_2(2,\Z)$ is generated by symbols. The only possibly non-trivial symbol is $c(-1,-1)$. Thus we re-derive the well-known fact that $K_2(2,\Z)$ is cyclic and generated by $c(-1,-1)$.
(In fact, $c(-1,-1)\in K_2(2,\Z)$ has infinite order.)
\end{exa}

\begin{exa}\label{exa:pi123}
 The calculations of Morita \cite{morita:k2zs} show that for  the Euclidean domains $A=\Z[\frac{1}{2}]$ or  $\Z[\frac{1}{3}]$, $K_2(2,A)$ is generated by symbols.
\end{exa}

\begin{exa}\label{exa:pi1m}
Furthermore, by \cite[Proposition 2.13]{morita:k2zs}, if $A$ is a Dedekind domain and if $K_2(2,A)$ is generated by symbols, then the same is true of $K_2(2,A[\frac{1}{\pi}])$ for any 
prime element $\pi$ of $A$ for which the homomorphism $D^\times\to (D/\an{\pi})^\times$ is surjective. 

It follows that $K_2(2,\Z[\frac{1}{m}])$ is generated by symbols whenever $m$ can be expressed as a product of primes $m=p_1^{a_1}\cdots p_t^{a_t}$ ($a_i\geq 1$) with the property
that $(\Z/p_i)^\times$ is generated by the residue classes $\{ -1, p_1,\ldots,p_{i-1}\}$ for all $i\leq t$. (In particular, $p_1\in\{ 2,3\}$).
\end{exa}

When $p\geq 5$ is a prime, however, the situation is quite different. The following result, as we explain in the proof, is essentially due to J. Morita (\cite{morita:braid},\cite{morita:mab}). 

\begin{lem}\label{lem:morita}  Let $p\geq 5$. Then $K_2(2, \Z[\frac{1}{p}])\not=C(2,\Z[\frac{1}{p}])$. 
More precisely, 
write $p=6k+\epsilon$ where $\epsilon\in \{ \pm1\}$. Furthermore, let $k=2^{\ell}m$ where $m$ is odd. Then the Dennis-Stein symbol
 $D(a,b):=D(-\epsilon\cdot 2^{\ell+1},3m)\in K_2(2,\Z[\frac{1}{p}])$ represents an element of infinite order in $K_2(2,\Z[\frac{1}{p}])/C(2,\Z[\frac{1}{p}])$.

 In fact, $D(a,b)$  represents an element of infinite order in $\left(K_2(2,\Z[\frac{1}{p}])/C(2,\Z[\frac{1}{p}])\right)^{\mathrm{ab}}$.
\end{lem}

\begin{proof}  Let $\mathrm{St}(2,\Z[\frac{1}{p}])$ denote the rank one Steinberg group (see the appendix). For a group $G$, let $G^{\mathrm{mab}}$ denote $G/G^{(2)}$, where 
$G^{(1)}:=[G,G]$ and $G^{(2)}:=[G^{(1)},G^{(1)}]$.  For $p\geq 5$, let $M_p$ denote the group given by the presentation 
\[
M_p=\an{\sigma,\tau_1,\tau_2\ |\ \sigma^{p^2-1}=[\tau_1,\tau_2]=1, \sigma\tau_1\sigma^{-1}=\tau_1\tau_2^{-1}, \sigma\tau_2\sigma^{-1}=\tau_1}
\]
Morita (\cite{morita:braid}) has shown that there is an isomorphism $\mathrm{St}(2,\Z[\frac{1}{p}])^{\mathrm{mab}}\cong M_p$. Now, by definition,  the subgroup of $M_p$ generated by 
$\tau_1$ and $\tau_2$ is isomorphic to $\Z\oplus \Z$ and the element $\sigma$ operates by conjugation as right multiplication by the matrix $
\left[
\begin{array}{cc}
1&1\\
-1&0\\
\end{array}
\right]=E(1)$. Thus $\sigma^{3}$ is multiplication by $-1$ and the element $\sigma^6$ is central in $M_p$. 
Since $\sigma$ has no fixed points in $\an{\tau_1,\tau_2}\cong \Z\oplus \Z$,  it follows that 
the centre, $Z(M_p)$, of $M_p$  is the cyclic group of order $(p^2-1)/6$ generated by $\sigma^6$.  However, in \cite[proof of Theorem 1]{morita:mab}, Morita shows that 
kernel of the composite homomorphism
\[
 M_p \cong \mathrm{St}\left(2,\Z\left[\frac{1}{p}\right]\right)^{\mathrm{mab}} \to \spl{2}{\Z\left[\frac{1}{p}\right]}^{\mathrm{mab}}
\]
is contained in the abelian group 
\[
M_p^0:=\an{\tau_1,\tau_2,\sigma^{12}} \cong \Z\oplus\Z\oplus \left(\Z/\left((p^2-1)/{12}\right)\right).
\]
Thus the image of the natural map $K_2(2,\Z[\frac{1}{p}])\to M_p$ lies in $M_p^0$. Any central element 
of $\mathrm{St}(2,\Z[\frac{1}{p}])$ must have image in $Z(M_p)=\an{\sigma^6}$. Thus the image of \\
 $C(2,\Z[\frac{1}{p}])\subset K_2(2,\Z[\frac{1}{p}])\cap Z(\mathrm{St}(2,\Z[\frac{1}{p}])$ in $M_p$ lies in $M_p^0\cap \an{\sigma^6}=\an{\sigma^{12}}$. 
It follows that there is a well-defined homomorphism 
\[
\frac{K_2\left(2,\Z\left[\frac{1}{p}\right]\right)}{C\left(2,\Z\left[\frac{1}{p}\right]\right)}\to
\left( \frac{K_2\left(2,\Z\left[\frac{1}{p}\right]\right)}{C\left(2,\Z\left[\frac{1}{p}\right]\right)}\right)^{\mathrm{ab}}\to 
\frac{M_p^0}{\an{\sigma^{12}}}\cong \an{\tau_1,\tau_2}\cong \Z\oplus \Z
\]
Finally, the calculations of Morita \cite[p74]{morita:mab} show that the image of $D(a,b)$ under this homomorphism is nontrivial (and hence has infinite order). 
\end{proof}
\begin{cor} For $p\geq 5$, $\Z[\frac{1}{p}]$ is a Euclidean domain which is not universal for $\mathrm{GE}_2$.
\end{cor}
\begin{exa}\label{exa:pi15}
Let $p\geq 5$ be a prime number.   Write $p=6k+\epsilon$ where $\epsilon\in \{ \pm1\}$. Let $k=2^{\ell}m$ where $m$ is odd. 
%The results of Morita in \cite{morita:mab} show that the Dennis-Stein symbol $D(-\epsilon\cdot 2^{\ell+1},3m)\in K_2(2,\Z[\frac{1}{p}])$ represents an element of infinite order in $K_2(2,\%Z[\frac{1}{p}])/C(2,\Z[\frac{1}{p}])$. 
%Thus these are examples of Euclidean domains which are not universal for $\mathrm{GE}_2$s.
By Lemma \ref{lem:morita},  the Dennis-Stein symbol   $D(-\epsilon\cdot 2^{\ell+1},3m)$ corresponds to (the class of) a loop in $\pi_1(Y(\Z[\frac{1}{p}),\infty)$ which has infinite order.  
We give an explicit expression for this loop in the next section.
\end{exa}

%%%%%%%%%%%%%%%%%%%%%%%%%%%%%%%%%%%%%%%%%%%%%%%%%%%%%%%%%%%%%%%%%%%%%%
\subsection{The loop in $Y(A)$ corresponding to a Dennis-Stein symbol}
%%%%%%%%%%%%%%%%%%%%%%%%%%%%%%%%%%%%%%%%%%%%%%%%%%%%%%%%%%%%%%%%%%%%%%%%%%%%%%%%%
Let $A$ be a ring and let $a,b\in A$ such that $u:=1-ab$ is a unit.  Then these determine the \emph{ Dennis-Stein symbol}
\[
D(a,b):=x_{21}(-bu^{-1})x_{12}(-a)x_{21}(b)x_{12}(au^{-1})h_{12}(u)^{-1}\in K_2(2,A).
\]
We now calculate the image of such an element under the composite map\\
 $K_2(2,A)\to K_2(2,A)/C(2,A)\to \pi_1(Y(A),\infty)$:

The map $\tilde{\alpha}$ (see Appendix \ref{sec:k2}) sends $D(a,b)$ to 
\[
y(-bu^{-1})\tilde{y}(-a)y(b)\tilde{y}(au^{-1})h(u^{-1})\in U(A),
\]
which, by definitions of the terms, is equal to
\[
\epsilon(0)^3\epsilon(-bu^{-1})\epsilon(a)\epsilon(0)^3\epsilon(0)^3\epsilon(b)\epsilon(-au^{-1})\epsilon(0)^3h(u^{-1}).
\]
Using the facts that $\epsilon(0)^2=h(-1)$, $h(-1)$ is central of order $2$ and $h(-1)h(u^{-1})=h(-u^{-1})$ in $C(A)$, this simplifies to:
\[
\epsilon(0)\epsilon(-bu^{-1})\epsilon(a)\epsilon(b)\epsilon(-au^{-1})\epsilon(0)h(-u^{-1}).
\]
The map $\Lambda$ in turn sends this to the element
\[
\an{1, (u^{-1},u,u^{-1},0,-au^{-1},b,a,-bu^{-1},0)}\in \pi_1(Y(A),\infty).
\]
Now applying a type (2) contraction to $u^{-1},u,u^{-1}\ldots$, this is in the same path homotopy class as
\[
\an{1, (0,0,0,-ua,u^{-2}b,u^2a,-bu^{-3},0)}.
\]
Applying two type (1) contractions, this is in the same class as $\an{1, (0,-ua,u^{-2}b,u^2a)}$.  The corresponding loop in $Y(A)$ is 
\begin{eqnarray*}
L_{a,b}&:=&(\infty,\infty\cdot E(0),\infty\cdot E(-ua)E(0),\infty\cdot E(u^{-2}b)E(-ua)E(0),\infty)\\
&=& (\infty,0,(ua)_-,(u^{-1}b)_+,\infty).
\end{eqnarray*}
When $A$ is an integral domain we thus have 
\[
L_{a,b}=(\infty,0,\frac{1}{ua},\frac{b}{u},\infty), 
\]
using the identification of Lemma \ref{lem:bezout}

Thus whenever $D(a,b)$ represents a nontivial element of $K_2(2,A)/C(2,A)$, the loop $L_{a,b}$ represents a nontrivial element of $\pi_1(Y(A),\infty)$ (and conversely).

\begin{exa}\label{exa:ds1} Let $p\geq 5$ and let $m,\ell$,$\epsilon$ be as in Example \ref{exa:pi15}. It follows that the loop
\[
L_{-\epsilon\cdot 2^{\ell+1},3m}=\left(\infty,0,-\frac{1}{2^{\ell+1}p},\frac{3m}{\epsilon p},\infty\right)
\]
represents an element of infinite order in $\pi_1(Y(\Z[\frac{1}{p}]),\infty)$.

For example, when $p=5$, we have $m=1$, $\ell=0$, $\epsilon=-1$ . The loop $(\infty,0,-\frac{1}{10},-\frac{3}{5},\infty)$ represents an element of infinite order 
in  $\pi_1(Y(\Z[\frac{1}{5}]),\infty)$.
\end{exa}
%%%%%%%%%%%%%%%%%%%%%%%%%%%%%%%%%%%%%%%%%%%%%%%%%%%%%%%%%%%%%%%%%%%%
\section{Application: The complex $L_\bullet(A)$}
%%%%%%%%%%%%%%%%%%%%%%%%%%%%%%%%%%%%%%%%%%%%%%%%%%%%%%%%%%%%%%%%%%%%%%%
For a ring $A$, $L_\bullet(A)$ is the following complex of abelian groups. $L_n(A)$ is the free abelian group with basis $X_n(A)$ where 
$X_n(A):= \{ (x_0,\ldots, x_n))\in Y_0(A)^n\ |\ \{ x_0,\ldots,x_n\}\in Y_n(A)\}$; i.e., a basis of $L_n(A)$ consists of \emph{ordered} $(n+1)$-cliques of $\Gamma(A)$. The boundary map is the standard simplicial boundary
\[
d_n: L_n\to L_{n-1},\ d_n((x_0,\ldots, x_n)):=\sum_{i=0}^n(-1)^i(x_0,\ldots,\widehat{x_i},\ldots,x_n).
\]

More generally, for any simplicial complex $Y$ let us define the complex $L_\bullet(Y)$ as follows: $L_n(Y)$ is the free abelian group whose basis consists of ordered $n$-simplices 
of $\Gamma$, together with the above simplicial boundary map. Thus $L_\bullet(A)=L_\bullet(Y(A))$ in this terminology.
We wish  to compare this complex with the \emph{ordered chain complex} $\Delta_\bullet(Y)$, defined as follows: A basis of $\Delta_n(Y)$ consists of ordered 
$(n+1)$-tuples $(y_0,\ldots,y_n)$ where $\{ y_0,\ldots,y_n\}$ is a simplex of $Y$  dimension $\leq n$; i.e., we allow  \emph{degenerate simplices} for which $y_i=y_j$ for some  $i\not=j$. 
It is well known that $H_n(\Delta_\bullet(Y))\cong H_n(|Y|)$ for all $n\geq 0$ (Spanier \cite[Chapter 5]{spanier}\footnote{In \cite[Chapter 5, Section 3]{spanier} it is not clear from the definition that the complex $\Delta_\bullet(Y)$ includes degenerate simplices. However, they are certainly required in the proof of Theorem 5.3.6, which in turn is required implicitly in the proof, via acyclic models, that the homology groups of $\Delta_\bullet(Y)$ agree with  those of the more standard \emph{oriented chain complex}, $C_\bullet(Y)$, of $Y$, which computes $H_\bullet(|Y|)$. }). 

The obvious chain map $L_\bullet(Y)\to \Delta_\bullet(Y)$ does not generally induce an isomorphism on homology: Clearly $H_0(L_\bullet(Y))=H_0(\Delta_\bullet(Y))$ for any simplicial complex $Y$.
However, if, for example,  $Y$ is an $n$-simplex then $H_i(\Delta_\bullet(Y))=0$ for all 
$i>0$ but $H_{n}(L_\bullet(Y)) $ is a free group of rank $(n+1)!\left(\frac{1}{0!}-\frac{1}{1!}+\cdots +\frac{(-1)^{n+1}}{(n+1)!}\right)$.

As observed above in section 
\ref{sec:clique}, 
any edge in $\Gamma(A)$ is contained in $3$-clique; equivalently, any $1$-simplex in $Y(A)$ is contained in a $2$-simplex. 

\begin{lem}\label{lem:h1} Let $Y$ be a simplicial set with the property that each $1$-simplex is contained in a $2$-simplex. Then the chain map $L_\bullet(Y)\to \Delta_\bullet(Y)$ induces 
an isomorphism 
\[
H_1(L_\bullet(Y))\cong H_1(\Delta_\bullet(Y)).
\]
\end{lem}
\begin{proof} Let $f=(f_n)_n$ be the chain map $L_\bullet\to \Delta_\bullet$.
We begin by observing that if $x$ is a $0$-simplex, then in $\Delta_\bullet$ we have $(x,x)=d_2(x,x,x)$. Thus if $z=\sum_{i}n_i(x_i,y_i)+\sum_jm_j(z_j,z_j)$ is a $1$-cycle in 
$\Delta_1$ , where $x_i\not=y_i$ then $d_1(\sum_{i}n_i(x_i,y_i))=0$ and $z=f_1(\sum_{i}n_i(x_i,y_i))$ modulo boundaries and thus $H_1(L_\bullet)$ maps onto $H_1(\Delta_\bullet)$. 

Now suppose that $z=\sum_in_i(x_i,y_i)\in L_1$ has the property that $z=d_2(w)$ for some $w\in \Delta_2$. We must prove that $z\in d_2(L_2)$. We can write 
\[
w=\sum_jm_j(a_j,b_j,c_j)+\sum_kt_k(r_k,s_k,r_k)+\sum_\ell w_\ell(u_\ell,u_\ell,v_\ell)+\sum_m c_m(d_m,e_m,e_m).
\]
where the $a_j,b_j,c_j$ are distinct from each other and the $r_k\not=s_k$. Now, for each $k$, choose $z_k$ such that $\{ r_k,s_k,z_k\}$ is a $2$-simplex in $Y$. Observing 
that $d_2(r_k,s_k,r_k) =d_2((r_k, s_k,z_k)+(s_k,r_k,z_k)-(r_k,r_k,s_k))$, we can replace the terms in the second sum with terms of the same type as those in the first and third sums. 
Thus, without loss of generality, we can suppose 
\begin{eqnarray*}
w&=&\sum_jm_j(a_j,b_j,c_j)+\sum_\ell w_\ell(u_\ell,u_\ell,v_\ell)+\sum_m c_m(d_m,e_m,e_m)\\
&=&w'+ \sum_\ell w_\ell(u_\ell,u_\ell,v_\ell)+\sum_m c_m(d_m,e_m,e_m)
\end{eqnarray*}

where the $a_j,b_j,c_j$ are distinct from each other; i.e., $w'\in L_2(Y)$. Since $d_2(u,u,w)=(u,u)$ and $d_2(d,e,e)=(e,e)$ we have 
\[
z=d(w)=d(w')+\sum_\ell w_\ell(u_\ell,u_\ell)+\sum_mc_m(e_m,e_m).
\]
Since $z,d_2(w')\in L_2(Y)$, it follows that $\sum_\ell w_\ell(u_\ell,u_\ell)+\sum_mc_m(e_m,e_m)=0$ and hence $z=d_2(w')$ with $w'\in L_2(Y)$ as required.
\end{proof}

\begin{thm}\label{thm:h1} Let $A$ be a ring. Then there are isomorphisms
\[
H_1(L_\bullet(A))\cong U(A)^{\mathrm{ab}}\cong \left(\frac{K_2(2,A)}{C(2,A)}\right)^{\mathrm{ab}}.
\]
\end{thm}
\begin{proof}
By Lemma \ref{lem:h1}, we have \[
H_1(L_\bullet(A))\cong H_1(\Delta_\bullet(Y(A)))\cong H_1(|Y(A)|)\cong \pi_1(Y(A),\infty)^{\mathrm{ab}}. 
\]
By Theorem \ref{thm:main} $\pi_1(Y(A),\infty)^{\mathrm{ab}}\cong U(A)^\mathrm{ab} \cong \left(K_2(2,A)/C(2,A)\right)^{\mathrm{ab}}$.
\end{proof}

\begin{cor} 
If $A$ is a $\mathrm{GE}_2$-ring which is universal for $\mathrm{GE}_2$, then $H_0(L_\bullet(A))\cong \Z$ and $H_1(L_\bullet(A))=0$.
Furthermore, if $K_2(2,A)$ is solvable the converse holds.
\end{cor}
\begin{proof} If $A$ is a $\mathrm{GE}_2$-ring  then $H_0(L_\bullet(A))\cong H_0(Y(A))\cong \Z[\pi_0(Y(A))]\cong \Z$ and if $A$ is universal for $\mathrm{GE}_2$ then 
$H_1(L_\bullet(A))\cong U(A)^{\mathrm{ab}}=1$.

If $K_2(2,A)$ is solvable, then so is $K_2(2,A)/C(2,A)$, and hence so is $U(A)$. In this case $U(A)=1$ if and only if $U(A)^{\mathrm{ab}}=1$.
\end{proof}

\begin{exa}\label{exa:lm}
$\Z[\frac{1}{m}]$ is  universal for $\mathrm{GE}_2$ if $m$ satisfies the condition described in Example \ref{exa:pi1m}. 
Thus
\[
H_0(L_\bullet(\Z\left[\frac{1}{m}\right]))=\Z\mbox{ and } H_1(L_\bullet(\Z\left[\frac{1}{m}\right]))=0
\]
for all such $m$. 
\end{exa}

\begin{exa}\label{exa:pi1p}
 Let $p\geq 5$ and let $m,\ell$,$\epsilon$ be as in Example \ref{exa:pi15}.  By Lemma \ref{lem:morita}, $\left(K_2(2,\Z[\frac{1}{p}])/C(2,\Z[\frac{1}{p}])\right)^{\mathrm{ab}}\not=0$
and hence $H_1(L_\bullet(\Z[\frac{1}{p}])\not=0$. More precisely, the image of the loop $L_{-\epsilon\cdot 2^{\ell+1},3m}$ in $L_1(\Z[\frac{1}{p}])$ represents a homology class of infinite order.  Thus, in particular,  
the cycle
\[
\left(\infty,0\right)+\left(0,-\frac{1}{2^{\ell+1}p}\right)+\left(-\frac{1}{2^{\ell+1}p},\frac{3m}{\epsilon p}\right)+\left(\frac{3m}{\epsilon p},\infty\right)\in 
L_1\left(\Z\left[\frac{1}{p}\right]\right)
\]
is not a boundary in the complex $L_\bullet(\Z[\frac{1}{p}])$. 
\end{exa}

%%%%%%%%%%%%%%%%%%%%%%%%%%%%%%%%%%%%%%%%%%%%%%%%%%%%%%%
%\vfill
%\pagebreak
\vskip 50pt
%%%%%%%%%%%%%%%%%%%%%%%%%%%%%%%%%%%%%%%%%%%%%%%%%%%%%%%%%%%%%%%%%%%%%
\appendix
%Appendix
%%%%%%%%%%%%%%%%%%%%%%%%%%%%%%%%%%%%%%%%%%%%%%%%%%%%%%%%%%%%
\section{Rings which are universal for $\mathrm{GE}_2$ and generation of $K_2(2,A)$ }\label{sec:k2}
%%%%%%%%%%%%%%%%%%%%%%%%%%%%%%%%%%%%%%%%%%%%%%%%%%%%%%%%%%%%%%%%%%%

It is well-known that for  rings $A$, the condition that $A$ is universal for $\mathrm{GE}_n$ is equivalent to the condition that $K_2(n,A)$ be generated by
symbols (see, for example, Dennis, Stein \cite{dennisstein:dvr}, p228). In \cite{dennisstein:dvr}, the authors refer to unpublished notes of Dennis. However, the author is not aware that the proof has been published. Since we rely on it above, we prove here (in the case $n=2$) the more precise statement that $U(A)\cong K_2(2,A)/C(2,A)$. The proof, of course, is essentially routine checking since both groups are given by presentations in terms of lifts of some elementary matrices.

We begin by reviewing the definitions and some of the basic properties of the rank one Steinberg group $\mathrm{St}(2,A)$ and the group $K_2(2,A)$. 

The rank one Steinberg group $\St{2}{A}$ is defined by generators and relations as follows:
The generators are the terms 
\[
x_{12}(t)\mbox{ and } x_{21}(t), \quad t\in A
\]
and the defining relations are 
\begin{enumerate}
\item 
\[
x_{ij}(s)x_{ij}(t)=x_{ij}(s+t)
\]
for $i\not=j\in \{ 1,2\}$ and all $s,t\in A$, and 
\item For $u\in A^\times$, let 
\[
w_{ij}(u):= x_{ij}(u)x_{ji}(-u^{-1})x_{ij}(u)
\]
for $i\not=j\in \{ 1,2\}$. Then 
\[
w_{ij}(u)x_{ij}(t)w_{ij}(-u)=x_{ji}(-u^{-2}t)
\]
for all $u\in A^\times$, $t\in A$.
\end{enumerate}

There is a natural surjective homomorphism $\phi:\st{2}{A}\to E_2(A)$ defined by 
$\phi(x_{ij}(t))=E_{ij}(t)$ for all $t$. 
 By definition, $K_2(2,A)$ is the kernel of $\phi$. 

For $u\in A^\times$ and for $i\not=j\in \{ 1,2\}$, we let 
\[
h_{ij}(u):= w_{ij}(u)w_{ij}(-1).
\]
Note that 
\[
\phi(w_{12}(u))=\matr{0}{u}{-u^{-1}}{0} \mbox{ and } \phi(h_{12}(u))=\matr{u}{0}{0}{u^{-1}}.
\]

Note that, from the definitions and defining relation (1), 
for any $a\in A$ and for any unit $u$ we have 
\[
x_{ij}(a)^{-1}=x_{ij}(-a)\mbox{ and } w_{ij}(u)^{-1}=w_{ij}(-u). 
\]

The defining relation  (2) above thus immediately gives the following 
conjugation formula. 
 
\begin{lem}\label{lem:conjw}
Let $A$ be a  ring. Let $a\in A$ and $u\in A^\times$. For $i\not=j\in 
\{ 1,2\}$ 
\[
x_{ij}(a)^{w_{ij}(-u)}=x_{ji}(-u^{-2}a).
\]
\end{lem}

Since the right-hand-side is unchanged when $-u$ is substituted for  $u$, we deduce:
\begin{cor}\label{cor:conjw}
Let $A$ be a  ring. Let $a\in A$ and $u\in A^\times$. For $i\not=j\in 
\{ 1,2\}$ 
\[
x_{ij}(a)^{w_{ij}(u)^{-1}}=x_{ji}(-u^{-2}a)=x_{ij}(a)^{w_{ij}(u)}.
\]
and 
\[
x_{ji}(a)^{w_{ij}(u)}=x_{ij}(-u^2a).
\]
\end{cor}

From the definition of $h_{ij}(u)$, we then obtain:
\begin{cor}\label{cor:conjh}
Let $A$ be a  ring. Let $a\in A$ and $u\in A^\times$. For $i\not=j\in 
\{ 1,2\}$ 
\[
x_{ij}(a)^{h_{ij}(u)}=x_{ij}(u^{-2}a)\mbox{ and }
x_{ij}(a)^{h_{ij}(u)^{-1}}=x_{ij}(u^2a).
\]
\end{cor}

\begin{cor}\label{cor:wij}
Let $A$ be a ring.  Let $u,v\in A^\times$ and let $i\not= j\in \{ 1,2\}$.
\begin{enumerate}
\item
$w_{ij}(u)^{w_{ij}(v)}=w_{ji}(-v^{-2}u)$.
\item 
$w_{ij}(u)=w_{ji}(-u^{-1})$.
\end{enumerate}
\end{cor}
\begin{proof}\ 

\begin{enumerate}
\item This follows immediately from Corollary \ref{cor:conjw} and the definition of the $w_{ij}(u)$.
\item Let $u=v$ is part (1).
\end{enumerate}
\end{proof}

\begin{cor}\label{cor:hij}
Let $A$ be a ring. Then for all $u\in A^\times$ and $i\not=j\in\{ 1,2\}$ we have 
\[
h_{ji}(u)=h_{ij}(u)^{-1}\mbox{ in }\St{2}{A}.
\]
\end{cor}
\begin{proof} Let $u\in A^\times$. Then
\begin{eqnarray*}
h_{ij}(u)^{-1}&=&w_{ij}(-1)^{-1}w_{ij}(u)^{-1}\\
&=&w_{ij}(1)w_{ij}(-u)\\
&=&w_{ij}(-u)^{w_{ij}(1)^{-1}}w_{ij}(1)\\
&=&w_{ji}(u)w_{ji}(-1)=h_{ji}(u)\mbox{ by Corollary \ref{cor:wij}}.\\
\end{eqnarray*}
\end{proof}

For $u,v\in A^\times$ the \emph{symbols} 
\[
c(u,v):= h_{12}(u)h_{12}(v)h_{12}(uv)^{-1}
\]
clearly lie in $K_2(2,A)$. 
The elements $c(u,v)$ are central in $\st{2}{A}$ by Corollary \ref{cor:conjh}. We let $C(2,A)$ denote the 
subgroup of $K_2(2,A)$ generated by these symbols. 

Our goal is to prove that $C(A)$ is isomorphic to $\St{2}{A}/C(2,A)$ and to deduce that $U(A)\cong K_2(2,A)/C(2,A)$. 

We first define the elements $\bar{y}(a):= \epsilon(-a)\epsilon(0)^3\in C(A)$ for any $a\in A$.  Thus $\psi(\bar{y}(a))=E_{12}(a)$ and,  analogously to Lemma \ref{lem:ya} for the elements $y(a)$, we have
\begin{lem}\label{lem:bary}
 Let $A$  be a ring.
\begin{enumerate}
\item $\bar{y}(a)\bar{y}(b)=\bar{y}(a+b)$ in $C(A)$  for all $a,b\in A$.
\item  For all $u\in A^\times$, $a\in A$ we have 
\[
\bar{y}(a)^{h(u)}=\bar{y}(u^{-2}a) \mbox{ in }C(A).
\]
\end{enumerate}
\end{lem}

We also require
\begin{lem}\label{lem:eps} For all $a\in A$, 
\[
y(a)^{\epsilon(0)}=\bar{y}(-a)\mbox{ and }  \bar{y}(a)^{\epsilon(0)}={y(-a)}.
\]
\end{lem}

\begin{proof} Let $a\in A$. Then
\begin{eqnarray*}
y(a)^{\epsilon(0)}&=& \epsilon(0)^{-1}\epsilon(0)^3\epsilon(a)\epsilon(0)\\
&=& \epsilon(0)^{3}\epsilon(0)^3\epsilon(a)\epsilon(0)= h(-1)\epsilon(a)\epsilon(0)\\
&=& \epsilon(a)h(-1)\epsilon(0)=\epsilon(a)\epsilon(0)^3=\bar{y}(-a).\\
\end{eqnarray*}
The second identity is proved similarly.
\end{proof}

For $u\in A^\times$ we define the elements of $C(A)$
\[
w(u):=\bar{y}(u)y(-u^{-1})\bar{y}(u)\mbox{ and }\bar{w}(u):=y(u)\bar{y}(-u^{-1})y(u).
\]

\begin{lem}\label{lem:wu}. Let $u\in A^\times$. Then in $C(A)$ we have:
\begin{enumerate}
\item $w(u)=h(u)\epsilon(0)$ and $\bar{w}(u)=\epsilon(0)h(-u)$.
\item $w(-u)=w(u)^{-1}$ and $\bar{w}(-u)=\bar{w}(u)^{-1}$.
\end{enumerate}
\end{lem}
\begin{proof} We prove the first identity in each case. The second is entirely similar. 
\begin{enumerate}
\item Let $u\in A^\times$.
\begin{eqnarray*}
w(u)=\bar{y}(u)y(-u^{-1})&=& \epsilon(-u)\epsilon(0)^3\epsilon(0)^3\epsilon(-u^{-1})\epsilon(-u)\epsilon(0)^3\\
&=& \epsilon(-u)\epsilon(-u^{-1})\epsilon(-u)\epsilon(0)\mbox{ since $\epsilon(0)^2=h(-1)$ and $h(-1)^2=1$}\\
&=& h(u)\epsilon(0).
\end{eqnarray*}
\item For $u\in A^\times$ 
\begin{eqnarray*}
w(u)w(-u)=h(u)\epsilon(0)h(-u)\epsilon(0)=\left(h(u)\epsilon(0)h(u)\right)h(-1)\epsilon(0)=\epsilon(0)\epsilon(0)^2\epsilon(0)=1.
\end{eqnarray*}
\end{enumerate}
\end{proof}

\begin{cor}\label{cor:wu} For all $a\in A$, $u\in A^\times$, 
\[
y(a)^{\bar{w}(-u)}=\bar{y}(-u^{-2}a)\mbox{ and } \bar{y}(a)^{{w}(-u)}={y}(-u^{-2}a)
\]
\end{cor}
\begin{proof}
We prove the first identity. Let $a\in A$ and $u\in A^\times$.
\begin{eqnarray*}
y(a)^{\bar{w}(-u)}&=&
% y(a)^{w(-u)}\mbox{ by Lemma \ref{lem:wu} (2)}\\
%&=& 
\left(y(a)^{\epsilon(0)}\right)^{h(u)}\mbox{ by Lemma \ref{lem:wu} (1)}\\
&=& \bar{y}(-a)^{h(u)}  \mbox{ by Lemma \ref{lem:eps}}\\
&=& \bar{y}(-u^{-2}a) \mbox{ by Lemma \ref{lem:bary} (2) }.\\
\end{eqnarray*}
\end{proof}

In view of the definition of $\St{2}{A}$, Lemmas \ref{lem:ya}and \ref{lem:bary}, the definition of $w(u)$ and $\bar{w}(u)$ and Corollary \ref{cor:wu} we deduce:

\begin{prop}\label{prop:alpha} Let $A$ be a ring. There is a well-defined group homomorphism 
\[
\alpha:\St{2}{A}\to C(A),\quad x_{21}(a)\mapsto y(a),\  x_{12}(a)\mapsto \bar{y}(a).
\]
Furthermore $\psi\circ\alpha=\phi:\St{2}{A}\to \spl{2}{A}$.
\end{prop}

\begin{cor}\label{cor:alpha}
$C(2,A)\subset \ker{\alpha}$ and hence $\alpha$  induces a homomorphism\\
 \[
\bar{\alpha}:\St{2}{A}/C(2,A)\to C(A). 
\]
\end{cor}
\begin{proof}
Note that $\alpha(w_{12}(u))=w(u)$  for any $u\in A^\times$. Thus 
\begin{eqnarray*}
\alpha(h_{12}(u))=\alpha(w_{12}(u)w_{12}(-1))=w(u)w(-1)=h(u)\epsilon(0)h(-1)\epsilon(0)=h(u)\epsilon(0)^4=h(u).
\end{eqnarray*}
It follows that for any $u,v\in A^\times$
\[
\alpha(c(u,v))=\alpha(h_{12}(u)h_{12}(v)h_{12}(uv)^{-1})=h(u)h(v)h(uv)^{-1}=1
\]
by relation (3) in the definition of $C(A)$. 
\end{proof}

Now let $\tilde{C}(A)$ denote the free group on generators $\tilde{\epsilon}(a)$, $a\in A$. Define the group homomorphism  $\tilde{\gamma}:\tilde{C}(A)\to \St{2}{A}/C(2,A)$ by
$\tilde{\gamma}(\tilde{\epsilon}(a)):= w_{21}(-1)x_{21}(a)$ for $a\in A$. We will show that $\tilde{\gamma}$ induces a well-defined homomorphism $C(A)\to \St{2}{A}/C(2,A)$.

We will require the following:

\begin{lem}\label{lem:w-1}
Let $A$ be a ring.  In $\St{2}{A}/C(2,A)$ we have $w_{ij}(-1)^4=w_{ij}(1)^4=1$ for $i\not=j\in \{ 1,2\}$.
\end{lem}

\begin{proof} In $\st{2}{A}$, we have $h_{ij}(1)=w_{ij}(1)w_{ij}(-1)=w_{ij}(1)w_{ij}(1)^{-1}=1$. Thus $c(-1,-1)=h_{12}(-1)^2=w_{12}(-1)^4=1$ in 
 $\St{2}{A}/C(2,A)$ . Since $w_{21}(1)=w_{12}(-1)=w_{12}(1)^{-1}$ the result follows.
\end{proof}

\begin{lem}\label{lem:tildeh}
 Let $u\in A^\times$. Let $\tilde{h}(u):=\tilde{\epsilon}(-u)\tilde{\epsilon}(-u^{-1})\tilde{\epsilon}(-u)\in \tilde{C}(A)$. Then 
\[
\tilde{\gamma}(\tilde{h}(u))=h_{12}(u)\mbox{ in } \St{2}{A}/C(2,A).
\]
\end{lem}

\begin{proof} We have in $\St{2}{A}/C(2,A)$:
\begin{eqnarray*}
\tilde{\gamma}(\tilde{h}(u))&=&\tilde{\gamma}(\tilde{\epsilon}(-u)\tilde{\epsilon}(-u^{-1})\tilde{\epsilon}(-u))\\
&=&
w_{21}(-1)x_{21}(-u)w_{21}(-1)\cdot x_{21}(-u^{-1})w_{21}(-1)\cdot x_{21}(-u)\\
&=&w_{21}(-1)x_{21}(-u)w_{21}(-1)\cdot w_{21}(-1) x_{21}(-u^{-1})^{w_{21}(-1)}\cdot x_{21}(-u)\\
&=&w_{21}(-1)x_{21}(-u)w_{21}(-1)^2 x_{12}(u^{-1})x_{21}(-u)\\
&=&w_{21}(-1)^3\cdot x_{21}(-u) x_{12}(u^{-1})x_{21}(-u)\mbox{ since $w_{21}(-1)^2=h_{21}(-1)$ is central}\\
&=&w_{21}(-1)^{-1}w_{21}(-u)
=w_{21}(-u)^{w_{21}(-1)}w_{21}(-1)^{-1}\\
&=&w_{12}(u)w_{12}(-1)=h_{12}(u).\\
\end{eqnarray*}
\end{proof}

\begin{thm}\label{thm:gamma}
The map $\tilde{\gamma}:\tilde{C}(A)\to  \St{2}{A}/C(2,A)$ induces a well-defined isomorphism $\gamma:C(A)\to  \St{2}{A}/C(2,A)$. Furthermore, restriction of $\gamma$ to 
$U(A)$ induces an isomorphism 
\[
U(A)\cong \frac{K_2(2,A)}{C(2,A)}.
\]
\end{thm}

\begin{proof} 
To establish the existence and well-definednes of $\gamma$, we show that $\tilde{\gamma}$  vanishes on the three families  of defining relations for $C(A)$:

Let $u, v\in A^\times$. Then 
\[
\tilde{\gamma}(\tilde{h}(u)\tilde{h}(v))=h_{12}(u)h_{12}(v)=c(u,v)h_{12}(uv)=h_{12}(v)=\tilde{\gamma}(\tilde{h}(uv))\mbox{ in } \St{2}{A}/C(2,A).
\]

Let $a,b\in A$. Then in $\St{2}{A}/C(2,A)$ we have
\begin{eqnarray*}
\tilde{\gamma}(\tilde{\epsilon}(a)\tilde{\epsilon}(0)\tilde{\epsilon}(b))&=&w_{21}(-1)x_{21}(a)w_{21}(-1)w_{21}(-1)x_{21}(b)\\
&=&w_{21}(-1)^3x_{21}(a)x_{21}(b)=w_{21}(-1)^3x_{21}(a+b)\\
&=&w_{21}(-1)^2\cdot w_{21}(-1)x_{21}(a+b)\\
&=& h_{21}(-1)\cdot w_{21}(-1)x_{21}(a+b)=h_{12}(-1)\cdot w_{21}(-1)x_{21}(a+b)\\
&=& \tilde{\gamma}(\tilde{h}(-1)\tilde{\epsilon}(a+b)).\\
\end{eqnarray*}

Let $u\in A^\times$, $a\in A$.  In $\St{2}{A}/C(2,A)$ we have
\begin{eqnarray*}
\tilde{\gamma}(\tilde{h}(u)\tilde{\epsilon}(a)\tilde{h}(u))&=&h_{12}(u)w_{21}(-1)x_{21}(a)h_{12}(u)\\
&=& w_{21}(-1)h_{12}(u)^{w_{21}(-1)}x_{21}(a)h_{12}(u)\\
&=&w_{21}(-1)h_{21}(-u)x_{21}(a)h_{12}(u)\\
&=&w_{21}(-1)h_{12}(u)^{-1}x_{21}(a)h_{12}(u)\\
&=&w_{21}(-1)x_{21}(a)^{h_{12}(u)}= w_{21}(-1)x_{21}(u^2a)\\
&=&\tilde{\gamma}(\tilde{\epsilon}(u^2a))
\end{eqnarray*}
as required. 

Thus we have a well-defined homomorphism $\gamma:C(A)\to \St{2}{A}/C(2,A)$ given by $\epsilon(a)\mapsto w_{21}(-1)x_{21}(a)$. We verify that the maps $\gamma$ and $\bar{\alpha}$ 
are inverse to each other: 

For any $a\in A$, 
\begin{eqnarray*}
\gamma(\bar{\alpha}(x_{12}(a)))=\gamma(\bar{y}(a))&=&\gamma(\epsilon(-a)\epsilon(0)^3)\\
&=&w_{21}(-1)x_{21}(-a)w_{21}(-1)^3\\
&=&x_{21}(-a)^{w_{21}(-1)^{-1}}=x_{12}(a),
\end{eqnarray*}
and
\begin{eqnarray*}
\gamma(\bar{\alpha}(x_{21}(a)))=\gamma({y}(a))&=&\gamma(\epsilon(0)^3\epsilon(a))\\
&=&w_{21}(-1)^3\cdot w_{21}(-1)x_{21}(a)\\
&=& w_{21}(-1)^4x_{21}(a)=x_{21}(a).
\end{eqnarray*}

Conversely, for all $a\in A$,
\begin{eqnarray*}
\bar{\alpha}(\gamma(\epsilon(a)))=\bar{\alpha}(w_{21}(-1)x_{21}(a))=\bar{w}(-1)y(a)=\epsilon(0)\epsilon(0)^3\epsilon(a)=\epsilon(a)
\end{eqnarray*}
as required, since $\bar{w}(-1)=\epsilon(0)h(1)=\epsilon(0)$ by Lemma \ref{lem:wu} (1).

Thus $\bar{\alpha}$ and $\gamma$ are mutually inverse isomorphisms. Since $\psi\circ\bar{\alpha}=\bar{\phi}:\St{2}{A}/C(2,A)\to \spl{2}{A}$ it follows that 
restriction of $\gamma$ induces an isomorphism 
\[
U(A)=\ker{\psi}\cong \ker{\bar{\phi}}=\frac{K_2(2,A)}{C(2,A)}.
\]

\end{proof}

%%%%%%%%%%%%%%%%%%%%%%%%%%%%%%%%%%%%%%%%%%%%%%%%%%%%%%%%%%%%%%%%%%%%%%%%%%%%%%%%%%%%%%%%%%

%%%%%%%%%%%%%%%%%%%%%%%%%%%%%%%%%%%%%%%%%%%%%%%%%%%%%%%%%%%%%%%%%%%%%%%%%%%%%%%%

\end{document}